\documentclass{amsart}

\usepackage{amssymb, hyperref, url, xcolor, tikz-cd}

\newtheorem{theorem}{Theorem}[section]
\newtheorem{proposition}[theorem]{Proposition}
\newtheorem{lemma}[theorem]{Lemma}
\newtheorem{corollary}[theorem]{Corollary}
\newtheorem{fact}[theorem]{Fact}

\theoremstyle{definition}
\newtheorem{definition}[theorem]{Definition}

\theoremstyle{remark}
\newtheorem{remark}[theorem]{Remark}
\newtheorem{example}[theorem]{Example}

\numberwithin{equation}{section}

\def\Ind{\setbox0=\hbox{$x$}\kern\wd0\hbox to 0pt{\hss$\mid$\hss} \lower.9\ht0\hbox to 0pt{\hss$\smile$\hss}\kern\wd0} 

\def\Notind{\setbox0=\hbox{$x$}\kern\wd0\hbox to 0pt{\mathchardef \nn=12854\hss$\nn$\kern1.4\wd0\hss}\hbox to 0pt{\hss$\mid$\hss}\lower.9\ht0 \hbox to 0pt{\hss$\smile$\hss}\kern\wd0}

\def \U {\mathcal U}

\def \eq {\operatorname{eq}}
\def \tp {tp}
\def \defin {\operatorname{def}}
\def \HH {\operatorname{H}}

\def \diff {\operatorname{diff}}
\def \alg {\operatorname{alg}}
\def \aut {\text{Aut}}
\def \mcl {\operatorname{mcl}}
\def \acl {\operatorname{acl}}
\def \dcl {\operatorname{dcl}}

\def \DCF {\operatorname{DCF}}

\title[Definable Galois cohomology]{On definable Galois theory and definable Galois cohomology in the totally transcendental setting}

\author{David Meretzky}

\address{David Meretzky\\ Universidad de Los Andes\\ Departamento de Matematicas\\ Carrera 1 \#18A - 12\\
Edificio H \\ Bogotá - Colombia 111711} 
\email{davmicmar@gmail.com}
\date{\today}

\begin{document}

\begin{abstract}
    	This paper gives results which relate the definable Galois theory of \cite{OmarAnand} to the definable Galois cohomology of \cite{pillay1997} in the setting of a totally transcendental first order theory. Firstly, we show that under some common assumptions the definable Galois cohomology of the extrinsic definable Galois group associated to fixed internality data classifies the number of definable Galois extensions inside a copy of the prime model. This generalizes a well-known result for Picard-Vessiot differential Galois theory shown originally via tannakian methods \cite{DeligneMilne2022}. 

        We then collate some triviality results for differential Galois cohomology \cite{PILLAY2017809} and recent results on iterated Picard-Vessiot extensions \cite{magid2022completepicardvessiotclosure} \cite{meretzky2026galoistheoryautomorphismgroups} to give colimit formulas for differential Galois cohomology. Precisely, for an ordinary differential field $K$ of characteristic $0$ and a linear differential algebraic group $G$ over $K$, the differential Galois cohomology $\text{H}^1_{\delta}(K,G)$, is given by a colimit over the family of normal closures of iterated Picard-Vessiot extensions of $K$ of finite type.  We then propose a new notion of boundedness for a differential field.
        
        We finally show that the minimal closure of a set of parameters $A$, the intersection of all elementary embeddings of a copy of the prime model over $A$ into itself, contains all of the definable Galois cohomological information for definable groups over $A$ satisfying also some strong but common conditions. Lastly, we give some colimit formulas for definable Galois cohomology and propose a model-theoretic definition of a bounded set of parameters. 
\end{abstract}

\maketitle

\section{Introduction}

In this paper we give two kinds of results which relate the definable Galois cohomology of \cite{pillay1997} to the definable Galois theory of \cite{OmarAnand} in the setting of the prime model of a totally transcendental theory. Both sets of results draw inspiration from the theory of differentially closed fields of characteristic $0$, where we look at interactions between the Picard-Vessiot theory and (constrained) differential Galois cohomology as presented in \cite{kolchin1985differential}. In this setting we give a new colimit formula for differential Galois cohomology and propose a new definition of boundedness for a differential field.  

Our differential algebraic results continue a line of work \cite{pillay_1998_DGI} 
\cite{Pillay2004}
\cite{PILLAY2017809}  \cite{Minchenko_2019} 
\cite{Chatzidakis2017GeneralizedPE}
\cite{AnandDavidComm}  giving fundamental triviality and boundedness results for Galois cohomology in the setting of a differential field $K$ of characteristic $0$. We begin by reviewing the basic facts in the algebraic setting which this work generalizes. 

Boundedness is an important field arithmetic property which plays a role in existence results for Picard-Vessiot extensions over differential fields with non-algebraically closed fields of constants \cite{KamenskyPillay}. We recall the definition: A field $F$ is bounded if it has finitely many Galois extensions of each finite degree.

Galois cohomology is a fundamental tool in algebraic number theory \cite{Serre1979}. Let $G$ be an algebraic group defined over a perfect field $F$. Then usual algebraic Galois cohomology $\HH^1_{\alg}(F,G)$ is a pointed set classifying $F$-isomorphism classes of algebraic torsors for $G$. The basepoint corresponds to the class of $G$ considered as a torsor. 

There are a pair of fundamental relationships between torsors and Galois extensions which provide intuition: Firstly, a Galois extension of a field $F$ gives rise to a torsor for the Galois group. Secondly, any torsor for an algebraic group becomes isomorphic to the group acting on itself over some finite Galois extension. 

The close connection between torsors and Galois extensions tells us first that Galois cohomology can be understood as a limit of finite extensions:

\begin{fact}\label{fact: alg colim}\cite{Serre1979}
    Let $F$ be a perfect field and let $G$ be an algebraic group defined over $F$. Then $$\varinjlim \HH^1_{\alg}(E/F,G(E)) \cong \HH^1_{\alg}(F,G)$$ where the limit is taken over the directed system of finite Galois extensions $E$ of $F$ and $\HH^1_{\alg}(E/F,G(E))$ denotes the isomorphism classes of $G$-torsors which have an $E$-point. 
\end{fact} 

It also tells us that algebraic closure can be detected cohomologically:

\begin{fact}\cite{Serre1979}\label{fact: alg colim}
    A perfect field $F$ is algebraically closed if and only if $\HH^1_{\alg}(F,G) \cong 1$ for all linear algebraic groups $G$ defined over $F$. 
\end{fact}

Lastly, it tells us that boundeness can be characterized cohomologically: 

\begin{fact}\label{fact: alg bdd}\cite{Serre1979}
    Let $F$ be a perfect field. The following are equivalent
    
    \begin{enumerate}
        \item $F$ is bounded. 
        \item For any linear algebraic group $G$ defined over $F$, $\HH^1_{\alg}(F,G)$ is finite. 
        \item For any linear algebraic group $G$ defined over $F$ there is a finite Galois extension $E$ over $F$ such that $\HH^1_{\alg}(E/F,G(E))  \cong  \HH^1_{\alg}(F,G).$
    \end{enumerate}
\end{fact} 

In the final two sections of the paper we give versions of these statements in the setting of differential fields of characteristic $0$ and in the setting of  totally transcendental first order theories. To state our results we first review the conventions from \cite{meretzky2026galoistheoryautomorphismgroups} for the Picard-Vessiot theory and say a little bit about differential Galois cohomology.

Let $K$ be an ordinary differential field of characteristic $0$ with constant field $C_K$. In any differential field extension $L$ of $K$, the solutions to a homogeneous linear differential equation $\mathcal{L}(y) = y^{(n)} + a_{n-1}y^{(n)} + ... +a_0y = 0$ with $a_i \in K$, form a $C_L$-vector space of dimension $n$ at most in $L$. A fundamental system of solutions to an equation $\mathcal{L}(y) = 0$ is a set of $n$-independent solutions over the constant field of any differential field containing them. A Picard-Vessiot extension of finite type is an extension $L/K$ generated by a fundamental system of solutions to a homogeneous linear equation over $K$ such that $C_L=C_K$.  A Picard-Vessiot extension is a possibly infinite union (really, compositum in the differential closure $K^{\diff}$) of Picard-Vessiot extensions of finite type.

The constrained differential Galois cohomology was introduced in the differential setting by Kolchin and Kovacic after a period of subsequent refinements \cite{kolchin1985differential} and is equivalent to the definable Galois cohomology specialized to the theory of differentially closed fields of characteristic $0$, $\DCF_0$. This cohomology theory measures isomorphism classes of differential algebraic torsors for differential algebraic groups (equivalently definable groups in $DCF_0$) as exposited in \cite{kolchin1985differential}. We review definable Galois cohomology below in Section \ref{sec: Gal coho}.

Indeed versions of the basic relationships between Galois extensions and torsors from the algebraic setting extend to the differential setting. This takes the form of the well known ``torsor theorem" from Picard-Vessiot theory \cite{vdPS2003}. Namely, for a homogeneous linear differential equation $\mathcal{L}(y) = 0$ over a differential field $K$ the set of (realizations of a type of) fundamental systems of solutions in $K^{\diff}$ over $K(C_{K^{\diff}})$ form a right torsor for the Galois group attached to the extension (generated by any realization). This theorem and it's converse, namely that any Picard-Vessiot extension can be generated by a suitable equation on a certain torsor, were generalized to the totally transcendental setting in \cite{OmarDavid2026}. We quote the result as Theorem \ref{thm: cohomological fact} at the end of Section \ref{sec: def Gal th}. 

The analogue of the algebraic triviality result, Fact \ref{fact: alg coh triv} above, in the differential setting is from \cite{PILLAY2017809}. This theorem was recently generalized to the setting of difference Picard-Vessiot theory in \cite{Bachmayr_Wibmer_2026} and earlier partially to the multiple derivation setting in \cite{Chatzidakis2017GeneralizedPE}. A version of the same result also appears in \cite{Minchenko_2019}:

\begin{theorem}\label{thm: diff triviality}\cite{PILLAY2017809}
    Let $K$ be a differential field of characteristic $0$. Then $K$ is algebraically closed and closed under Picard-Vessiot extensions if and only if for any linear differential algebraic group $G$ over $K$, 
    $\HH_{\delta}^{1}(K,G)$ is trivial.
\end{theorem}

Our main contribution to this line of work in the differential field setting are differential versions of the algebraic colimit formula of Fact \ref{fact: alg colim}. A version of this fact for constrained cohomology appears in \cite{kolchin1985differential} but we find it unsatisfactory because the colimit is taken essentially over the family of all finitely differentially generated extensions of $K$ in $K^{\diff}$. 

Before stating the analogue of Fact \ref{fact: alg colim}, we point out two well known differences between the algebraic and differential algebraic settings. Firstly, the union of all Picard-Vessiot extensions $(K)^{PV} = K^{PV_{1}}$ of $K$ in $K^{\diff}$ is not the differential closure, $K^{\diff}$, and indeed is not even PV-closed. One must iterate this operation to obtain a sequence of PV-extensions $$K \subseteq K^{PV} = K^{PV_1} \subseteq (K^{PV_1})^{PV} = K^{PV_2} \subseteq ... $$ whose union $K^{PV_\infty}$ is closed under PV extensions and is most of the time still not equal to $K^{\diff}$. The theorem above, Theorem \ref{thm: diff triviality}, can then be reformulated as in \cite{AnandDavidComm} saying that any torsor for a linear differential algebraic group over $K$ has a $K^{PV_{\infty}}$-point:

\begin{theorem}\label{thm: diff triviality in PV}\cite{AnandDavidComm}
    Let $K$ be a differential field of characteristic $0$. Then for any linear differential algebraic group $G$ over $K$, 
    $\HH_{\delta}^{1}(K^{PV_\infty}/K,G(K^{PV_\infty})) \cong \HH_{\delta}^{1}(K,G)$.
\end{theorem}

A second difference between the two settings is that the differential closure of a differential field may properly embed into itself over $K$. We let $\mcl(K)$ denote the intersection of all self embeddings breaking from notation in \cite{meretzky2026galoistheoryautomorphismgroups}. We then have a sequence of differential extensions of $K$ all of which generalize, in different senses, the algebraic closure of a field:

$$K \subseteq K^{PV_1}\subseteq K^{PV_2} \subseteq ...\subseteq  K^{PV_\infty} \subseteq \mcl(K) \subseteq K^{\diff}$$

Our first (and essentially tautological) observation is that:

\begin{proposition}
    Let $G$ be a linear differential algebraic group defined over $K$ with $C_K = (C_K)^{\alg}$. Then $$\varinjlim H^1_{\defin}(L/K,G(L)) \cong H^1_{\defin}(K^{PV_1}/K,G(K^{PV_1}))$$ where the limit is again taken over the family of Picard-Vessiot extensions of finite type. 
\end{proposition}

In \cite{magid2022completepicardvessiotclosure} and  \cite{meretzky2026galoistheoryautomorphismgroups} a Galois correspondence for $K^{PV_\infty}$ is introduced and refined. It relates normal closures of iterated Picard-Vessiot extensions of finite type to relatively type definable subgroups of $\aut(K^{PV_\infty}/K)$. In Section \ref{sec: colim diff}, using Theorem \ref{thm: diff triviality in PV} together with results of \cite{magid2022completepicardvessiotclosure} and  \cite{meretzky2026galoistheoryautomorphismgroups} we give the following colimit description of $H^1(K,G)$ for a linear differential algebraic group $G$ over $K$. Namely, we give a differential analogue of Fact \ref{fact: alg colim}:

\begin{theorem}
    Let $G$ be a linear differential algebraic group defined over a differential field $K$ with $C_K = (C_K)^{\alg}$. Then $$\varinjlim\HH^1_{\delta}(L/K,G(L)) \cong \HH_{\delta}^{1}(K^{PV_\infty}/K,G(K^{PV_\infty}))$$ where the limit is taken over the directed system of normal closures of iterated PV extensions of $K$ of finite type. Note $\HH_{\delta}^{1}(K^{PV_\infty}/K,G(K^{PV_\infty}))\cong \HH^1_{\delta}(K,G)$ by Theorem \ref{thm: diff triviality in PV}.
\end{theorem}

At the end of Section \ref{sec: colim diff} we then propose the following definition of differential boundedness (giving a definitional form of Fact \ref{fact: alg bdd}) adding to a list of possible such definitions from \cite{AnandDavidComm}:

\begin{definition}\label{def: diff bdd}
    A differential field $K$ with $C_K = (C_K)^{\alg}$ is said to be bounded with respect to a differential algebraic group $G$ defined over $K$ if there is an iterated PV extension of finite type, whose normal closure $L$, satisfies $$\HH^1_{\defin}(L/K,G(L)) \cong \HH^1_{\defin}(K,G).$$
\end{definition}

We now discuss generalizations of our colimit and boundedness results in the totally transcendental setting which we present in Section \ref{sec: tt setting}.

In \cite{poizatimaginare}, Poizat gave the first model theoretic treatment of the Picard-Vessiot theory using the machinery of the binding group associated to internality data. This came after an investigation of the structure and minimality properties of prime models of totally transcendental first order theories \cite{PoizatPrime}.  The differential closure $K^{\diff}$ of a differential field being an example of the prime model of a totally transcendental theory, in this case $DCF_0$. In \cite{PoizatPrime}, the minimal closure of a set of parameters $A$ inside a copy of the prime model $M$ is defined as the intersection of all $A$-elementary embeddings $M$ into itself. We denote this $\mcl(A)$. The key property for Galois theory which $\mcl(A)$ enjoys is that any subset $B$ such that $A \subseteq B \subseteq \mcl(A) \subseteq M$, $\dcl(B)$ is exactly the fixed points of the action of $\aut(M/B)$ on $M$. Equivalently, $M$ is strongly homogeneous over any such $B$. In \cite{meretzky2026galoistheoryautomorphismgroups}, a Galois correspondence between $\dcl$-closed subsets of $\mcl(A)$ and relatively type definable subgroups of $\aut(M/A)$ is shown which in the case of differential fields gives the correspondence of \cite{magid2022completepicardvessiotclosure}.  In Section \ref{sec: min} we review the the minimal closure $\mcl(A)$ inside of  the prime model $M$ over a set of parameters $A$ in a totally transcendental theory. 

The first result of Section \ref{sec: tt setting} is the following lemma which appears as Lemma \ref{lem: mcl points}:

\begin{lemma}\label{lem: mcl points}
    Let $G$ be an $A$-definable group. Let $P$ be an $A$-definable PHS for $G$. If $G(M) = G(\mcl(A))$ then $P(M) = P(\mcl(A))$. 
\end{lemma} 

For our notion of definable Galois extension we follow \cite{OmarAnand}. We review this in Section \ref{sec: def Gal th}. Namely over a set of parameters $A$, with respect to a definable set $X$, a definable Galois extension $B\ = dcl(A,b)$ for $b$ a realization of a type strongly internal and weakly orthogonal to $X$. 

In Section \ref{sec: tt setting}, we define, following \cite{meretzky2026galoistheoryautomorphismgroups}, an iterated definable Galois extension $A \subseteq B \subseteq M$ relative to $X$ is a $\dcl$-closed subset such that either there exists an infinite sequence of $\dcl$-closed subsets $A=B_0 \subseteq B_1 \subseteq \cdots $ with each $B_i$ a proper definable Galois extension of $B_{i-1}$ such that $B= \cup_iB_i$, or there exists a finite such sequence $A=B_0 \subseteq B_1 \subseteq \cdots B_n=B$. If $B$ is the union of a finite such chain where each $B_i/B_{i-1}$ is a definable Galois extension of finite type, we call $B$ an iterated definable Galois extension of finite type of $A$ relative to $X$. Similarly we define the sequence of definable Galois closures $A \subseteq A^{XDG_1} \subseteq A^{XDG_2} \subseteq ... \subseteq A^{XDG_\infty}$ to be the sequence of definable Galois closures with reference to some definable set $X$ over $A$ in analogy with the sequence of Picard-Vessiot closures above. 

Using the Lemma \ref{lem: mcl points} we prove the following: 

\begin{theorem}\label{thm: mcl points}
    Let $G$ be an $A$-definable group with $G(M) = G(\mcl(A))$. Then
    \begin{enumerate}
        \item The definable Galois cohomology of $G$ over $\mcl(A)$ is trivial: $$\HH^1_{\defin}(M/\mcl(A),G(M)) = 1.$$
        \item Using the short exact sequence $$\HH^1_{\defin}(\mcl(A)/A,G(\mcl(A))) \cong \HH^1_{\defin}(A,G)$$ Namely, all of the definable Galois cohomological information for $G$ in $M$ is contained in $\mcl(A)$.
        \item Any definable Galois extension relative to $G$ is contained in $\mcl(A)$.
    \end{enumerate}
\end{theorem}

By the last part we obtain:

\begin{proposition}
    Let $A$ be a $\dcl$-closed set of parameters. Let $X$ be an $A$-definable set such that $X(M) = X(\mcl(A))$. Then each $A^{X\text{DG}_n}$ is contained in $\mcl(A)$ and we have $$A \subseteq A^{X\text{DG}_1} \subseteq A^{X\text{DG}_2} \cdots \subseteq A^{X\text{DG}_\infty} \subseteq \mcl(A) \subseteq M.$$ Furthermore each $A^{X\text{DG}_n}$ is normal.
\end{proposition}

In Section \ref{sec: Gal coho} we recall the definition of definable Galois cohomology from \cite{pillay1997}. For an $A$-definable group $G$ in a totally transcendental theory eliminating imaginaries, the definable Galois cohomology classifies
$A$-definable right principal homogeneous spaces (PHSs) or torsors for $G$. We use the short exact sequence of \cite{MERETZKY_SES} to prove the following general lemma: 

\begin{lemma}
    Let $M$ be a copy of the prime model over $A$. Let $B$ be normal in $M$ over $A$. Let $\{B_i\}_{i \in I}$ be a family of normal subsets of $B$ which form a directed system under inclusion and such that $\cup_{i \in I}B_i = B$. Then $\varprojlim \aut(B_i/A) \cong \aut(B/A)$ and moreover if $G$ is an $A$-definable group then $$\varinjlim \HH^1_{\defin}(B_i/A,G(B_i)) \cong \HH^1_{\defin}(B/A,G(B)).$$
\end{lemma}

We then obtain the following colimit formula in the totally transcendental setting, the condition $X(M)=X(A)$ analogous to the assumption that the constants be algebraically closed in the differential setting:

\begin{theorem}\label{thm: idg colim}
    Let $G$ be an $A$ definable group living on an $A$-definable set $X$ with $X(M)=X(A)$. Then $$\varinjlim \HH^{1}_{\defin}(B/A, G(B)) \cong \HH^{1}_{\defin}(A^{XDG_\infty}/A, G(A^{XDG_\infty}))$$ where the limit is taken over the family of normal closures of iterated definable Galois extensions of finite type $B$ of $A$ relative to $X$.
\end{theorem}

In contrast with the differential case we do not have a colimit formula for the full extension $M/A$. We leave to future work an investigation of the extensions $A^{XDG_\infty} \subseteq \mcl(A)$ for various sets $X$. Nevertheless we feel the results in this section, together with the correspondence of \cite{meretzky2026galoistheoryautomorphismgroups}, to be a meaningful response to the call in \cite{PoizatPrime} for an investigation of the automorphism group of $\mcl(A)$ over $A$. 

At the end of Section \ref{sec: tt setting}, we propose a definition of boundedness for a set of parameters:

\begin{definition}
    Let $X$ be an $A$-definable set satisfying $X(A)= X(M)$. Let $G$ be an $A$-definable group living on $X$.  The set of parameters $A$ is said to be bounded with respect to $G$ and $X$ if there exists an iterated definable Galois extension of finite type relative to $X$ whose normal closure $B$ satisfies $$\HH^{1}_{\defin}(B/A, G(B)) \cong \HH^{1}_{\defin}(A^{XDG_\infty}/A, G(A^{XDG_\infty})).$$
\end{definition}

In Section \ref{sec: mult} we give the other main result of the paper which is a generalization of the following result relating the multiplicity of Picard-Vessiot extensions to the algebraic Galois cohomology of the Galois group. There is no assumption here that the constant field be algebraically closed:

\begin{theorem}\label{fact: counting PV}\cite{KamenskyPillay}\cite{DeligneMilne2022} Let $\mathcal{L}(y)=0$ be an  homogeneous linear ordinary differential equation (LODE) over a differential field $K$ of characteristic $0$. Assume that there is at least one nontrivial Picard-Vessiot extension $L/K$ for $\mathcal{L}(y)=0$ in $K^{\diff}$, a differential closure of $K$. Let $G(C_{K^{\diff}})$ be the Picard-Vessiot differential Galois group of the extension $L/K$, a linear algebraic group defined over $C_K$. Then the usual algebraic Galois cohomology $$\HH^1_{\alg}(C_K,G)$$
classifies (it is a pointed set whose cardinality counts) exactly the set of distinct Picard-Vessiot extensions for $\mathcal{L}(y)=0$ in $K^{\diff}$. The base point of the cohomology set corresponds to the extension $L/K$.
\end{theorem}

The generalization of Theorem \ref{fact: counting PV} shown in Section \ref{sec: mult} is the following:

\begin{theorem} Let $T$ be a totally transcendental theory eliminating imaginaries. Let $A$ be a $\dcl$-closed set of parameters and $M$ be a copy of the prime model over $A$. Let $Y$ and $X$ be $A$-definable sets with $Y$ internal to $X$. Produce $f:Z\times X_1 \to Z$ as in the construction of the binding groupoid (given in detail in Section \ref{sec: binding}) where $Z$ is the $A$-definable set of ``fundamental systems" associated to $Y$, $X_1$ is an $A$-definable set contained in $\dcl^{eq}(A,X)$, and $f$ is an $A$-definable function such that for any $b_1,b_2 \in Z$, there is a unique $c \in X_1$ with $f(b_1,c) = b_2$. Then assuming

\begin{enumerate}
    \item there is at least one $b_0 \in Z(M)\backslash Z(A)$ with $\dcl(A,b_0) \cap \dcl(A,X) = A$ so that $A \subsetneq \dcl(A,b_0)$ is a proper definable Galois extension in the sense of \cite{OmarAnand} with extrinsic definable Galois group $G_0$,
    \item the $A$-definable set $X_1$ carries an $A$-definable group structure such that $Z$ is exactly an $A$-definable right PHS for $X_1$ with action $f$,
    \item and $\HH^1_{\defin}(M/A,X_1(M))$ is trivial
\end{enumerate}

we have that $\HH^1_{\defin}(M/A,G_0(M))$ classifies exactly the set of definable Galois extensions in $M$ for the internality data $(Y,X)$. The base point of the cohomology set corresponds exactly to the extension $A \subsetneq \dcl(A,b_0)$. 

\end{theorem}

In Section \ref{sec: mult} we comment in more detail on the various assumptions. It is stated in a slightly more general form where we also consider the case $\HH^1_{\defin}(M/A,X_1(M))$ is nontrivial i.e., assumption 3) is dropped. In general, there seems to be no reason why assumption 2) needs to hold, i.e., $X_1$ needs to carry a group structure making $(Z,X_1)$ into an $A$-definable PHS with action $f$. However in practice this very often is the case. Moreover, from the perspective of a fixed extension, this can always be taken to be the case as can be seen by the main result of \cite{OmarDavid2026}. 

%In Section \ref{sec: internality} we set conventions and notation around the notion of internality. 

%In Section \ref{sec: binding} we review the construction of the Galois groupoid arising from internality data in the totally transcendental setting. We go through standard arguments defining the triple $(f, Z, X_1)$ from the internality data $(X,Y)$. We then proceed to construct the definable Groupoid action and intrinsic group which all together we refer to as the Galois groupoid.  

The forward direction this theorem in the setting of $DCF_0$, i.e. that under suitable hypotheses $\HH^1_{\defin}(M/A,G_0(M))$ bounds the number of definable Galois extensions (this direction not needing assumption 3), is used in the existence results for strongly normal extensions of \cite{KamenskyPillay}. Much work has been done to extend these results to $DCF_{0,m}$ \cite{omarronnie}, \cite{omaranandppv} \cite{OmarAnandDavid}, but the actual extension, even just with respect to the constants, is not fully complete at the time of this paper in the strongly normal setting of $DCF_{0,m}$. We hope that the above theorem will find use extending these existence results to the generalized strongly normal theory, but the result should be of independent interest.

One can also see by inspection that the forward direction follows from the construction of the binding groupoid which can be seen as a generalization of the ``torsor theorem" of differential Galois theory which we give a detailed account of in Section \ref{sec: binding}.  The reverse direction is a straightforward application of the long exact sequence in cohomology developed in \cite{OmarAnandDavid}. The proof here is an adaptation the proof of Theorem \ref{fact: counting PV} via definable Galois cohomology given in Chapter 3 of \cite{MeretzkyThesis}. The original proof goes back to at least \cite{DeligneMilne2022} and is shown via tannakian methods. That there should be a generalization of this fundamental Theorem \ref{fact: counting PV} via the binding-groupoid should maybe be unsurprising given the close connection between the tannakian formalism and the model-theoretic notion of the binding groupoid associated to internality data \cite{moshetannak}, \cite{KamenskyPillay}.

 %(Aside: What if anything can be said when the internality data is assumed to be related to the covering map context of Scanlon?) 

%In the Picard-Vessiot setting where the internality data is exactly a homogeneous LODE, the set of fundamental systems in $K^{\diff}$, $Z(K^{\diff})$ form a right principal homogeneous space for $GL_n(C_{K^{\diff}}) = X_1(K^{\diff})$.

%In \cite{AnandDavidComm} we proposed some analogous definitions suitable for differential fields. 

\section{Acknowledgments} 

I would like to thank Anand Pillay for his generous advising and for introducing me to the circle of ideas above. Much of the material here appears in some form in my thesis completed under his direction. I want to thank Omar León Sánchez for his encouragement and advice and for arranging a visit to the University of Manchester Summer 2025 where I began writing this document. I would like to thank the model theory group at la Universidad de los Andes for reading through this, and related material, as part of the model theory seminar in Fall 2025. I would like to thank Pablo Cubides-Kovacics for his careful reading and comments around the material of Section \ref{sec: mult}. No AI was used in this project at any stage.

\section{Prime models, the minimal closure, and Galois correspondences}\label{sec: min}

In this section we review and collate some key definitions and constructions. 

Let $T$ be a complete totally transcendental (t.t.) theory eliminating imaginaries. Let $\bar{M}$ be a monster model of $T$ which we will assume to be $\kappa$-saturated and strongly $\kappa$-homogeneous for a suitably sized cardinal $\kappa$. Let $A$ be a small set of parameters.

As $T$ is t.t., the isolated types in $S_n(A)$ are dense and consequently constructible models $M$ over $A$ exist. By the back-and-forth argument of Resayrre, constructible models are unique up to $A$-isomorphism and strongly $\omega$-homogeneous over $A$ i.e. for tuples $a$ and $b$ from $M$ with the same type over $A$, there is an automorphism $\sigma \in \aut(M/A)$ taking $a$ to $b$. Constructible models are also atomic over $A$. Shelah proved the following characterization of constructibility:  In a t.t. theory a model $M$ over $A$ is constructible if and only if it is atomic over $A$ and has no uncountable $A$-indiscernible sequence. It is clear that a constructible model is prime over $A$, and as we said, in the t.t. setting constructible models exist. A prime model $N$ over $A$ must embed in a constructible model over $A$, it is then clear that $N$ satisfies the criterion of Shelah's characterization and so is then itself constructible. The uniqueness atomicity, and strong $\omega$-homogeneity of the prime model then follow. Alternatively one can show directly (without Shelah's characterization) that an arbitrary subset of a constructible model is still constructible in the t.t. setting.

Let $M$ be a copy of the prime model over $A$ of $T$. An intermediate extension $A \subseteq B \subseteq M$, is called normal if $\forall b \in B$ and $b' \in M$, $\tp(b/A) = \tp(b'/A)$ implies that $b' \in B$. By the strong $\omega$-homogeneity of $M$ over $A$, we have a short exact sequence of groups $1 \to \aut(M/B) \to \aut(M/A) \to \aut(B/A) \to 1$. Finally, we remark that $M$ remains prime over $B$. 

Although the prime model over a set of parameters of a t.t. theory is unique up to isomorphism over that set, it is not necessarily minimal, i.e., it may properly elementarily embed into itself. 

\begin{definition}\cite{PoizatPrime}
    Let $M$ be a copy of the prime model over $A$. Recall, the minimal closure of $A$ in $M$ is the intersection $\cap_{\gamma}\gamma(M)$ where $\gamma:M \to M$ ranges over all elementary embeddings of $M$ into itself.  We denote this $\mcl(A)$ in analogy with the model theoretic definable and algebraic closures, denoted $\dcl$ and $\acl$, respectively. A type $p \in S(A)$ is called atomically algebraic if it is isolated but has no infinite $A$-indiscernible set of realizations which is atomic over $A$. An element $a \in \bar{M}$ is atomically algebraic over $A$ if $\tp(a/A)$ is atomically algebraic. 
\end{definition}

\begin{fact}\cite{PoizatPrime}\label{fact: mcl}
    Let $M$ be prime over $A$ in $\bar{M}$. 
    \begin{enumerate}
        \item $\mcl(A)$ is equal to the set of atomically algebraic elements in $M$. Consequently, $M = \mcl(A)$, i.e. $M$ is minimal, if and only if it contains no infinite atomic indiscernible sequence.
        \item $mcl(mcl(A)) = \mcl(A)$
        \item $M$ remains constructible and hence prime over any subset of $mcl(A)$. 
        \item For any subsets $B$ and $B'$ of $mcl(A)$, if $tp(B/A) = tp(B'/A)$, there is an automorphism $\sigma \in \aut(M/A)$ taking $B$ to $B'$.
    \end{enumerate}
\end{fact}

%\begin{example}
%    \begin{enumerate}
%        \item Linear ODEs.
%        \item Non-example of Rosenlicht.
%    \end{enumerate}
%\end{example}

We also then have now three types of intermediate extensions $A \subseteq B \subseteq M$ over which the prime model remains prime, and hence atomic and strongly $\omega$-homogeneous: Firstly, finitely $\dcl$-generated extensions, i.e. when $B = \dcl(A,b)$ for a finite tuple $b$. Secondly for normal extensions $B$. Thirdly for arbitrary subsets $B$ of $mcl(A)$.  The use of this is principally in the Galois definability lemma in the context of a prime model (see \cite{meretzky2026galoistheoryautomorphismgroups} for instance):

\begin{fact}\cite{meretzky2026galoistheoryautomorphismgroups}
Let $A \subseteq B \subseteq M$ be an intermediate extension over which $M$ remains prime. 

\begin{enumerate}
    \item Let $X$ be an $M$ definable set and let $X(M)$ be $\aut(B/A)$-invariant. Then $X$ is $B$-definable.
    \item The fixed points of the action of $\aut(M/B)$ on $M$, denoted $M^{\aut(M/B)}$ is exactly $\dcl(B)$.
\end{enumerate} 
\end{fact}

This fact underpins both definable Galois theory and definable Galois cohomology in the setting of a prime model of a totally transcendental theory:

We need a few more definitions from \cite{meretzky2026galoistheoryautomorphismgroups}.

\begin{definition}\cite{meretzky2026galoistheoryautomorphismgroups}
    Let $B$ be a $\dcl$-closed intermediate set $A \subseteq B \subseteq \mcl(A) \subseteq M$. Let $\bar{b}$ be an enumeration of $B$.  Let $p_0(\bar{x})$ be $\tp(\bar{b}/A)$, and $S_{p_0}(B)$ the space of extensions of $p_0$ to $B$. Let $X$ be a closed subset of $S_{p_0}(B)$ given by a set of formulas $\Sigma(\bar{x},\bar{b})$ and $Y$ be a clopen subset of $S_{p_0}(B)$ given by a formula  $\varphi(x,b)$ for some $b$, a finite tuple from $B$.
    \begin{enumerate}
        \item By $Y(\aut(B/A))$ we mean a subset of the form $$\{\sigma \in \aut(B/A) : \ \models \varphi(b,\sigma(b))\}$$ which we call a relatively definable subset of $\aut(B/A)$. 
        \item By $X(\aut(\bar{M}/A))$ we mean $$\{\sigma \in \aut(\bar{M}/A) : \ \models \Sigma(\bar{b},\sigma(\bar{b}))\}$$ and we call this a relatively type definable over $B$ subset of $\aut(\bar{M}/A)$.
        \item By $X(\aut(B/A))$ we mean $X(\aut(\bar{M}/A))|_{\aut(B/A)}$, i.e. a subset of $\aut(B/A)$ of the form $$\{\sigma \in \aut(B/A) : \ \models \Sigma(\bar{b},\sigma(\bar{b}))\}$$ which we call a relatively type definable over $B$ subset of $\aut(B/A)$.
    \end{enumerate}
    A relatively (type) definable subgroup of $\aut(B/A)$ is then a relatively (type) definable subgroup of $\aut(B/A)$.
\end{definition}

The following theorem generalizes the Galois correspondence for $\acl^{eq}(A)$ to $\mcl^{eq}(A)$.

\begin{theorem}\cite{meretzky2026galoistheoryautomorphismgroups}
    Let $B$ be a $\dcl$-closed normal intermediate set $A \subseteq B \subseteq \mcl(A) \subseteq M$. Then there is a Galois correspondence between $dcl$-closed subsets $C$ of $B$ containing $A$, and closed subgroups $H$ of $\aut(B/A)$. Namely, $$C \mapsto H_C:= \{\sigma\in \aut(B/A) : \ \sigma(c) = c \ \forall c \in C\}$$ and $$H \mapsto B^H:= \{b \in B : h(b)= b \  \forall h \in H\}.$$ Moreover, finitely $\dcl$-generated extensions correspond to relatively definable subgroups and normal intermediate extensions correspond to normal relatively type definable subgroups.  
\end{theorem}

\section{Definable Galois cohomology and $G$-primitives}\label{sec: Gal coho}

The only novelty in this section is the colimit formula in definable Galois cohomology.

Let $M$ be a copy of the prime model over $A$ and $B$ a normal intermediate extension of parameters. Let $G$ be an $A$-definable group.  

Recall from \cite{pillay1997}, \cite{MERETZKY_SES} that a map $\varphi:\aut(B/A) \to G(B)$ is called an $A$-definable cocycle if it satisfies the usual cocycle condition, $\forall \sigma,\tau \in \aut(B/A)$ $$\varphi(\sigma\tau) =\varphi(\sigma)\sigma(\varphi(\tau))$$ and there exists a tuple $b \in B$ and $A$-definable function $h(x,y)$ such that $\forall\sigma \in \aut(B/A)$ $$h(b,\sigma(b)) = \varphi(\sigma).$$ Two cocycles $\varphi$ and $\psi$ are cohomologous if there exists an element $g \in G(B)$ such that $\forall \sigma \in \aut(B/A)$ $\varphi(\sigma) = g^{-1}\psi(\sigma)\sigma(g)$. 

\begin{definition}\cite{MERETZKY_SES}
    The first definable Galois cohomology set $\HH^1_{\defin}(B/A,G(B))$ is defined to be the set of cohomology classes of definable cocycles $\varphi:\aut(B/A) \to G(B)$.  When $B=M$ we write $\HH^1_{\defin}(A,G)$ simply.
\end{definition}

An $A$-definable right principal homogeneous space (PHS) or torsor for $G$ is an $A$-definable set $X$ with a strictly transitive $A$-definable action by $G$. Recall from \cite{MERETZKY_SES} that $\HH^1_{\defin}(B/A,G(B))$ is isomorphic to (i.e. in basepoint preserving bijection with) the isomorphism classes of definable PHSs for $G$ which have a $B$-point.

The main available tools for working with these cohomology sets are the following two theorems, Theorem \ref{thm: SES} and \ref{thm: LES} below:

\begin{theorem}\label{thm: SES}\cite{MERETZKY_SES}
    Let $M$ be a copy of the prime model over $A$ and $A \subseteq B \subseteq C\subseteq M$ be intermediate $\dcl$-closed sets of parameters with $B$ and $C$ normal in $M$ over $A$. Let $G$ be an $A$-definable group. Then the short exact sequence of automorphism groups $$1 \to \aut(C/B) \to \aut(C/A) \to \aut(B/A) \to 1$$ induces a short exact sequence in definable Galois cohomology\footnote{Note that we ignore the transform action of $\aut(B/A)$ on the final term. See Definition 4.1 of \cite{MERETZKY_SES}} $$1 \to \HH^1_{\defin}(B/A,  G(B)) \to \HH^1_{\defin}(C/A, G(C)) \to \HH^1_{\defin}(C/B, G(C)).$$
\end{theorem}

We obtain from this theorem the following general colimit formula for definable Galois cohomology. We will use this lemma in the final section for specific normal extensions $B$ and directed systems.

\begin{lemma}\label{lem: colim}
    Let $B$ be normal in $M$ over $A$. Let $\{B_i\}_{i \in I}$ be a family of normal subsets of $B$ which form a directed system under inclusion and such that $\cup_{i \in I}B_i = B$. Then $\varprojlim \aut(B_i/A) \cong \aut(B/A)$ and moreover if $G$ is an $A$-definable group then $$\varinjlim \HH^1_{\defin}(B_i/A,G(B_i)) \cong \HH^1_{\defin}(B/A,G(B)).$$
\end{lemma}

\begin{proof}
    Note that $\varinjlim B_i$ is the disjoint union of the $B_i$ modulo the relation of eventual equality. Each $B_i$ includes into $B$ which induces a map from $\varinjlim B_i$ to $B$. It follows immediately from the definitions that this is well defined and injective. Surjectivity follows from the assumption that $\cup_{i \in I}B_i = B$. So there is a canonical bijection from $\varinjlim B_i$ to $B$.

    As $\{B_i\}_{i \in I}$ forms a directed system, $\{\aut(B_i/A)\}_{i \in I}$ forms an inverse system. For each index $i$, the quotient homomorphism $Aut(B/A) \to \aut(B_i/A)$ is surjective and the required compatibilities hold so that we have a map $Aut(B/A) \to \varprojlim\aut(B_i/A)$. Injectivity follows by noting that if two automorphism of $Aut(B/A)$ differ, then they differ on some $B_i$. To check surjectivity, let $(\sigma_i)_{i\in I} \in \varprojlim\aut(B_i/A)$. We define an automorphism of $B$ whose image is $(\sigma_i)_{i\in I}$. For $b \in B$, there is an $i \in I$ such that $b\in B_i$, so define $\sigma(b) = \sigma_i(b)$. To see that this is well defined independent of the choice of $B_i$, let $i,j \in I$ such that $b \in B_i \cap B_j$. Then there is some $B_k$ containing both $B_i$ and $B_j$. Then $\sigma_i(b) = \pi_{ki}(\sigma_k)(b) = \sigma_k(b) = \pi_{kj}(\sigma_k)(b) = \sigma_j(b)$ where $\pi_{ki}:\aut(B_k/A) \to \aut(B_i/A)$ and $\pi_{kj}:\aut(B_k/A) \to \aut(B_j/A)$ are the projections. As $\sigma$ restricts to an elementary permutation on each $B_i$, and $B = \cup_{i \in I}B_i$, $\sigma$ is an elementary permutation. So $\aut(B/A) \cong \varprojlim \aut(B_i/A)$.

    By Theorem \ref{thm: SES}, as the $\{\aut(B_i/A)\}_{i \in I}$ form an inverse system of groups, $\{\HH^{1}_{\defin}(B_i/A, G(B_i))\}_{i \in I}$ forms a directed system of pointed sets. The inclusion of each $B_i$ in $B$ induces again via Theorem \ref{thm: SES} a map from $\varinjlim \HH^1_{\defin}(B_i/A,G(B_i))$ to $\HH^{1}_{\defin}(B/A,G(B))$. As $\cup_{i\in I}B_i = B$, this map is surjective as the definability data for any cocycle in $\HH^{1}_{\defin}(B/A,G(B))$ must be contained in from some $B_i$. Injectivity is immediate: If for some indices  $i,j \in I$ there are class in $\HH^1_{\defin}(B_i/A,G(B_i))$ and $\HH^1_{\defin}(B_i/A,G(B_j))$ which have the same image in $\HH^1_{\defin}(B/A,G(B))$, then they must be equal in $\HH^1_{\defin}(B_k/A,G(B_k))$ where $B_i,B_j \subseteq B_k$. So $\varinjlim \HH^1_{\defin}(B_i/A,G(B_i)) \cong \HH^{1}_{\defin}(B/A,G(B))$.
    
\end{proof}

The following theorem of \cite{OmarAnandDavid} is the other fundamental method for computing with Galois cohomology in the definable setting.

\begin{theorem}\label{thm: LES}\cite{OmarAnandDavid}
     Let $H$ be an $A$-definable subgroup of a definable group $G$. We then have an exact sequence of pointed sets $$1 \to H(A) \to G(A) \to (G/H)(A) \xrightarrow{\delta} \HH^1_{\defin}(A,H)\to \HH^1_{\defin}(A,G)$$ where $(G/H)(A)$ is the pointed set of $A$-definable left cosets of $H$ in $G$.
\end{theorem}

The connecting homomorphism $\delta$ of the above theorem identifies the $A$-definable left cosets of $H$ in $G$ among all right $A$-definable PHSs for $H$. The primitive elements in case (1) of the following definition are exactly basepoints of $A$-definable PHSs for $H$ which are in the image of $\delta$. The primitive elements of case $(2)$ (for nontrivial $W$) are basepoints for $A$-definable PHSs for $H$ which lie outside of the image of $\delta$.

\begin{definition}\label{def: prim elt}
     Let $H$ be an $A$-definable subgroup of a definable group $G$. 
     \begin{enumerate}
         \item \cite{pillaysokolovic} Let $\pi:G\to G/H$ be the $A$-definable quotient map given by elimination of imaginaries. An element $a \in G$ is said to be a $(G,H)$-primitive over $A$ if $\pi(a) \in (G/H)(A)$. 
         \item \cite{MeretzkyThesis} Let $W$ be an $A$-definable right torsor for $G$ and $\pi:W \to W/H$ the $A$-definable quotient map given by elimination of imaginaries. An element $a \in W$ is said to be a $(W,G,H)$-primitive over $A$ if $\pi(a) \in (W/H)(A)$.
     \end{enumerate}
\end{definition}

\section{Internality}\label{sec: internality}

The material in this section is a standard part of geometric stability theory but we include it to be able to carefully state and prove the results of Section \ref{sec: mult} as conventions vary by author.

Let $T$ be a totally transcendental theory eliminating imaginaries. Let $\bar{M}$ be a monster model of $T$. Assume the language has at least two constants. Let $A$ be a small $\dcl$-closed set of parameters in $\bar{M}$. Let $M$ be the prime model of $T$ over $A$.

%Let $Y$ be an $A$-definable set in $\bar{M}$. If $Y$ is internal to an $A$-definable set $X$ then there are a few ways of viewing the group $Aut(Y/X,A)$ as a definable group. By $Aut(Y/X,A)$ we mean the elementary permutations of $Y$ which fix $X$ and $A$ pointwise. That is, permutations $\sigma$ of $Y$ such that for any $\aaa = (a_1,...,a_n) \in Y^n$, and $\varphi(\xx) \in L(X,A)$, $\bar{M} \models \varphi(\aaa)$ if and only if $\bar{M} \models \varphi(\sigma(\aaa))$.

\begin{definition}\label{internality}
Let $Y$ and $X$ be a $A$-definable sets. Let $\psi_Y(\bar{z})$ and $\varphi_X(\bar{w})$ be $A$-definable formulas defining $Y$ and $X$ respectively.  We say that $Y$ is internal to $X$ if there is a small subset of parameters $B \subset \bar{M}$ such that $Y \subseteq \dcl(B,X,A)$. 
\end{definition}

We set up notation:

\begin{remark}
     $Y$ is internal to $X$ means that, for each $\bar{a} \in Y$, there is an $A$-formula, $\varphi_{\bar{a}}(\bar{y}_{\bar{a}},\bar{x}_{\bar{a}},\bar{z})$, a tuple $\bar{b}_{\bar{a}} \in B^{|\bar{y}_{\bar{a}}|}$ (the variable lengths can change depending on the $\bar{a}$) and a tuple $\bar{c}_{\bar{a}} \in X^{|\bar{x}_a|}$ such that $\models \varphi_{\bar{a}}(\bar{b}_{\bar{a}},\bar{c}_{\bar{a}},\bar{a})$ and $\bar{a}$ is the unique element of $\bar{M}$ satisfying the above formula.
\end{remark}

\begin{lemma}\label{Y definability lemma}
Let $Y$ be $X$-internal as above. Then there exists an $A$-definable function $f(\bar{y},\bar{x})$ and a tuple of elements $\bar{b}$ from $Y$ such that for each $\bar{a} \in Y$, there exists a tuple $\bar{c}$ from $dcl(X,A)$ such that $f(\bar{b},\bar{c}) = \bar{a}$. 
\end{lemma} 

\begin{remark}
The tuple $\bar{c}$ in most cases will come from $X$. In the case that $X$ and $A$ are both empty, the assumption that the language has multiple constants guarantees that $dcl(X,A)$ is still at least countably infinite: Note that if there are at least two constants in the language then finite binary sequences may be coded.
\end{remark}

\begin{proof}
Note that $A$ and $B$ are small.  By assumption, the collection of $A\cup B$-formulas $$\{\psi_Y(\bar{z})\} \cup \{\neg \exists \bar{x}_{\bar{a}} \in X^{|\bar{x}_{\bar{a}}|}(\varphi_{\bar{a}}(\bar{b},\bar{x}_{\bar{a}},\bar{z})) : \bar{b} \in B^{\bar{y}_{\bar{a}}}\}_{\bar{a} \in Y}$$ is inconsistent. Therefore it is finitely inconsistent by compactness. Thus there are finitely many $\bar{a}_1,...,\bar{a}_m \in Y$ such that $$\models \forall \bar{z} \left( \psi_Y(\bar{z}) \rightarrow \bigvee_{i=1}^m \ \exists \bar{x}_{\bar{a}_i} \in X^{|\bar{x}_{\bar{a}_i}|} \  \varphi_{\bar{a}_i}(\bar{b}_{\bar{a}_i},\bar{x}_{\bar{a}_i},\bar{z})\right).$$ 

Let $e_1,...,e_m$ be distinct elements from $dcl(X,A)$ (which exist by the preceding remark) associated to variables $u_1,...,u_m$. Then for each $\bar{a} \in Y$, there exists $\bar{c}_{\bar{a}} \in X^{|\bar{x}_{\bar{a}_i}|}$ and $e_j \in dcl(X,A)$ such that $$\models\bigvee_{i=1}^m (\varphi_{\bar{a}_i}(\bar{b}_{\bar{a}_i},\bar{c}_{\bar{a}},\bar{a}) \wedge e_i=e_j)$$ and $\bar{a}$ is the only such solution.

Denote the above formula by $\varphi(\bar{y},\bar{x},\bar{z})$ where
we let $\bar{b} =(b_{\bar{a}_1},...,b_{\bar{a}_m})$ in the variables $\bar{y} = (\bar{y}_{\bar{a}_1},...,\bar{y}_{\bar{a}_m})$ and $\bar{x} = (\bar{x}_{\bar{a}_1},u_1,....,\bar{x}_{\bar{a}_m},u_m, u)$ with $u$ being the variable filled by the element $e_j$. It is clear that $\varphi(\bar{b}, \bar{x},\bar{z})$ defines the graph of a definable function.

Let $\bar{a} \in Y$ now and $\bar{c} \in dcl(X,A)$ such that $\models \varphi(\bar{b},\bar{c},\bar{a})$. Then $\varphi(\bar{y},\bar{c},\bar{a}) \in \tp(\bar{b}/dcl(X,A)\cup Y)$. As $T$ is totally transcendental, by definability of types over arbitrary subsets of $\bar{M}$, $\varphi(\bar{b},\bar{x},\bar{z})$ is definable over $dcl(X,A) \cup Y$. 

We now make some replacements in notation and we are finished. 

The set defined by $\varphi(\bar{b},\bar{x},\bar{z})$, is defined by a $dcl(X,A) \cup Y$-formula which we will now write as $\varphi_0(\bar{b},\bar{c}, \bar{x},\bar{z})$. Where $\bar{b}=(b_1,...,b_n)$ is a new tuple of parameters from $Y$, $\bar{c}$, in say variables $\bar{v}$, is a new tuple of parameters from $\dcl(A,X)$ and $\varphi_0(\bar{y},\bar{v},\bar{x}, \bar{z})$ is $A$-definable. As before, $\varphi_0(\bar{b},\bar{c},\bar{x}, \bar{z})$ is the graph of a definable function. Now both $\bar{v}$ and $\bar{x}$ are sorts on $\dcl(A,X)$ so rewrite them as just $\bar{x}$. Then $\varphi_0(\bar{b},\bar{x},\bar{z})$ is still clearly a the graph of a definable function $f(\bar{b},\bar{x}) = \bar{z}$.

The tuple $\bar{x} = (\bar{x}_{\bar{a}_1},u_1,....,\bar{x}_{\bar{a}_m},u_m,u,\bar{v})$. So the $\bar{x}$ variable ranges over an $A$-definable subset $X_0$ of $\dcl(A,X)$.
\end{proof}

\begin{definition}
We call the tuple $\bar{b} = (b_1,...,b_n)$ a fundamental system of solutions associated to $Y$, $X$.
\end{definition}

\begin{remark}
Tracing through the above proof we can see that if $X$ is the empty set, then $Y$ is algebraic over a finite extension of $A$ from $B$ i.e. $$\models \forall \bar{z}\left(\varphi_Y(\bar{z}) \to \bigvee_{i=1}^m\varphi_{\bar{a}_i}(\bar{b}_{\bar{a}_i},\bar{z}) \right)$$ with the disjunction being an algebraic formula. That is, $Y$ implies an algebraic formula over $\dcl(A,\bar{b}_{\bar{a}_1},...,\bar{b}_{\bar{a}_m})$. Conversely, any algebraic $A$-definable set $Y$ is internal to the $\emptyset$, because it is in the definable closure of $A$ union a full (finite) set of solutions to $Y$ in $\bar{M}$ which is finite and therefore small. That is, for any small $A$-definable set $Y$ (which then must be algebraic), $Y \subseteq \dcl(A,Y)$ witnesses internality.
\end{remark}

\begin{lemma}\label{lem: Z A def}
Let $Y$ be internal to $X$, both defined over a small $\dcl$-closed set $A$. The set of fundamental systems $Z \subset Y^n$ is $A$-definable.   
\end{lemma}

\begin{proof}
The $A$-definable formula $$\forall \bar{z} \in Y \exists \bar{x} \in X_0 \ f(\bar{y},\bar{x}) = \bar{z}$$ in the variable $\bar{y}$, defines the subset $Z \subseteq Y^n$. Note that this $n$ is the length of any fundamental system i.e., the $n$ in the proof of Lemma \ref{Y definability lemma}.
\end{proof}

\begin{remark}\label{fundsysinM}
As $M$ is an elementary substructure of $\bar{M}$, the previous lemma shows that we can find a fundamental system $\bar{b}$ in $M$, the prime model over $A$.
\end{remark}

\begin{lemma}\label{lem: xyz1f}
    Let $Y$ and $X$ be $A$-definable with $Y$ internal to $X$. Let $Z$ be the associated $A$-definable set of fundamental systems.  Then there is an $A$-definable subset $X_1$ of $dcl(A,X)$ and an $A$-definable function $f(\bar{y},\bar{x}): Z \times X_1 \to Z$ such that for any $\bar{b},\bar{b}' \in Z$ there is a unique $\bar{c} \in X_1$ with $f(\bar{b},\bar{c}) = \bar{b}'$.
\end{lemma}

\begin{proof}
    Let $b \in Y$. For any $\bar{b} \in Z \subseteq Y^n$ there is a $\bar{c} \in X_0$ such that $f(\bar{b},\bar{c}) = b$. 
    
    Then for any $\bar{b},\bar{b}' = (b_1',...,b_n') \in Z$ there is a tuple $(\bar{c}_1,...,\bar{c}_n) \in (X_0)^n$ such that $f(\bar{b},\bar{c}_i) = b_i'$. 
    
    Define an equivalence relation on $(X_0)^n$ by saying $\hat{c} = (\bar{c}_1,...,\bar{c}_n)$ and $\hat{c}' = (\bar{c}_1',...,\bar{c}_n')$ are equivalent if $$\forall \bar{y} \in Z  \ \bigwedge_{i = 1}^n f(\bar{y},\bar{c}_i) = f(\bar{y},\bar{c}_i').$$
    
    The quotient of $(X_0)^n$ by this $A$-definable equivalence relation gives another $A$-definable set $X_1$ living in $\dcl^{\eq}(X,A)$ which by EI is just $\dcl(X,A)$. We then replace $f(\bar{y}, \bar{x})$ in notation with $f(\bar{y}, \bar{x}_1),...,f(\bar{y}, \bar{x}_n)$ with the $\bar{x}_i$ ranging over $X_1$. With $f(\bar{y}, \bar{x}_i)$ being the induced function $Z\times X_1 \to Y$. 

    Now $f(\bar{y}, \bar{x})$ is then an $A$-definable function function $Z\times X_1 \to Y^n$. Lastly, by Lemma \ref{lem: Z A def} we replace $X_1$ in notation with the definable subset of itself so that the image of this map is exactly $Z$.
\end{proof}

\begin{example}
Let $Y$ be the set defined by $$y^{(n)} + a_{n-1}y^{(n-1)}+...+a_0y = 0,$$ an ordinary homogeneous linear differential equation in scalar form defined over a differential field $K$. 

The $K$-definable set $Y$ is internal to the constants as its solutions in $K^{\diff}$ form an $n$-dimensional vector space over the constants $C_{K^{\diff}}$. The set of such fundamental systems for $Y$ that is, the set of bases for the vector space of solutions, is a $K$-definable set $Z(K^{\diff})$. 

The set $Z(K^{\diff})$ is interdefinable with the set of fundamental matrices for $Y$ in $K^{\diff}$. We replace $Z(K^{\diff})$ by this set of matrices. It will be a right PHS for the group $GL_n(C_{K^{\diff}})$ which is $X_1$ in this example.
\end{example}

\section{Definable Galois extensions}\label{sec: def Gal th}

The notion of definable Galois extension which we use is the notion given in \cite{OmarAnand}. The only addition to the exposition here is we describe some properties of these extensions with respect to the minimal closure of a set of parameters inside of a prime model.

Let $T$ be a complete totally transcendental theory eliminating imaginaries. Let $\bar{M}$ be a monster model. Let $A$ be a small set of parameters. Let $X$ be an $A$-definable set. Let $M$ be a copy of the prime model over $A$.

\begin{definition}\cite{OmarAnand}
    A type $p(\bar{x})$ over $A$ is said to be strongly internal to a set $X$ if for any $b,b' \models p(\bar{x})$, $b' \in \dcl(A,X,b)$.
\end{definition}

\begin{example}\label{fund sys Y}
Let $Y$ be internal to $X$ as in the previous section. Then the type of any fundamental system for the data $(Y,X)$ is strongly internal to the set $X$ in the sense of \cite{OmarAnand}. 
\end{example}

\begin{definition}\cite{OmarAnand}\label{def: definable Galois extension}
    A type $p(\bar{x})$ over $A$ is said to be weakly orthogonal to $X$ if for any tuple $\bar{c}$ from $X$, letting $r(\bar{x}) = \tp(\bar{c}/A)$, $p(\bar{y})\cup r(\bar{x})$ extends to a complete type over $A$ in the variables $\bar{y}\bar{x}$. Equivalently, $p(\bar{x})$ has a unique extension to a type over (the not small set of parameters) $\dcl(A,X)$.
\end{definition}

\begin{fact}[\cite{OmarAnand} Lemma 2.2 i)]\label{weak orth fact}
    The type $q$ is weakly orthogonal to $X$ if and only if for any realization $b$ of $q$, $\dcl(A,b) \cap X(M)=X(\dcl(A,b))=X(A)$.
\end{fact}

\begin{definition}\cite{OmarAnand}\label{def: def gal ext}
    A definable Galois extension $A\subseteq B$ relative to $X$ is a $\dcl$-closed finitely $\dcl$-generated extension $B = \dcl(A,\bar{b})$ where $\bar{b}$ is a realization of a type $p(\bar{x})$ strongly internal and weakly orthogonal to $X$.
\end{definition}

\begin{fact}\label{fact: def Galois corresp.}
    Let $p(\bar{y})$ be a type over $A$ which is strongly internal and weakly orthogonal to a set $X$ defined over $A$. 
    \begin{enumerate}
        \item The type $p(\bar{y})$ is atomic, and so is realized in $M$. 
        \item If $X$ is strongly minimal, then $X(M) = X(\acl(A))$.
        \item If $X$ is strongly minimal, then every realization of $p(\bar{y})$ in $M$ is contained in the minimal closure of $A$ (the intersection of all embeddings of $M$ into itself.)
    \end{enumerate}
\end{fact}

\begin{proof} We give a mixture of references and proofs.
    \begin{enumerate}
        \item This is shown in \cite{OmarAnand} and \cite{pillay_1998_DGI}. We will essentially give the argument in the next section.
        \item This is Lemma 2.4 of \cite{pillay_1998_DGI}. By strong minimality, the only realizations of $X$ which are atomic are algebraic over $A$ hence $X(M) = X(\acl(A))$.
        \item This is in \cite{poizatimaginare}. Let $\bar{b_1}$ be a realization of the type $p(\bar{y})$ in $M$. Let $\gamma:M \to M$ be an elementary embedding. So $\gamma(M) \prec M$. Let $\bar{b}_0$ be a realization of $p(\bar{y})$ in $\gamma(M)$. Then $\bar{b}_1 \in \dcl(A,X,\bar{b}_0)$. It quickly follows that $\bar{b}_1 \in \dcl(A,X(M),\bar{b}_0) = \dcl(A,X(\acl(A)),\bar{b}_0)$. As $\acl(A)$ is contained in the minimal closure of $A$, $\dcl(A,X(\acl(A)),\bar{b}_0)$ is contained in the image of $\gamma$. So $b_1$ is in the image of all such embeddings. 
    \end{enumerate}
\end{proof}

We finish this subsection with the connection between model theoretic primitives and definable Galois extensions which we will use in Remark \ref{rmk: assmp ii}:

\begin{theorem}\label{thm: cohomological fact}\cite{OmarDavid2026} Let $\bar{M}$ be a monster model of a t.t. theory $T$ eliminating imaginaries, and let $A$ be a (small) set of parameters.
\begin{enumerate}
    \item Let $Q$ be an $A$-definable right torsor for an $A$-definable group $H$, and let $b \in Q(\U)$. Then, $B=\dcl(A,b)$ is a Galois extension of $A$ relative to $H$ in the sense of Definition \ref{def: def gal ext} if and only if $H(B)= H(A)$.
    
    \item Let $B/A$ be a definable Galois extension; namely, $B=\dcl(A,b)$ with $q=tp(b/A)$ strongly internal and weakly orthogonal to some $A$-definable set. Let $Q \circlearrowright H$ be the associated $A$-definable right torsor given by the torsor theorem (i.e., part (2) of Fact~\ref{fact: def Galois corresp.}). Suppose $H$ is subgroup of an $A$-definable group $G$. Then, there is an $A$-definable right torsor $W$ for $G$ such that $Q$ is $A$-definably isomorphic to an orbit of $H$ on $W$ and the Galois extension $B/A$ can be generated by a $(W,G,H)$-primitive over $A$.
        
        Moreover, the image of the class associated to the torsor $Q$ under the natural map 
        $$\HH^1_{\defin}(A,H) \xrightarrow{\iota^1} \HH^1_{\defin}(A,G)$$ is trivial if and only if $Q$ is $A$-definably isomorphic to an $A$-definable left coset of $H$ in $G$. In this case, the Galois extension $B/A$ can be generated by a $(G,H)$-primitive. This happens in particular when $\HH^1_{\defin}(A,G) = 1$.
\end{enumerate}
\end{theorem}

\section{The binding groupoid}\label{sec: binding}

Let $Y$ and $X$ be $A$-definable with $Y$ internal to $X$. Let $Z$ be the associated $A$-definable set of fundamental systems and by Lemma \ref{lem: xyz1f} let $f:Z\times X_1 \to Z$ be such that for any $b,b'\in Z$, there exists a unique $c\in X_1$ with $f(b,c) = b'$.

\begin{remark}
    Note that from any pair of elements of $Z$ we can define uniquely an element of $X_1$. We have a definable function $h:Z^2 \to X_1$. Then $f(y_1, h(y_2,y_3)) = y_4$ is a ternary function. If this ternary function satisfies two identities
    \begin{enumerate}
        \item $f(y_1,h(y_2,y_2) = y_1 = f(y_2,h(y_2,y_1))$
        \item $f(y_1,h(y_2,f(y_3,h(y_4,y_5)))) = f(f(y_1,h(y_2,y_3)),h(y_4,y_5))$
    \end{enumerate}
    we can recover uniquely the structure of a group on $X_1$ and a PHS on $Z$. There seems to be no reason in general for these identities to be satisfied, i.e. for $(Z,X_1)$ to have the structure of a PHS. At the end of this subsection we address the case where $(Z,X_1)$ is a PHS.
\end{remark}

\begin{lemma}\label{lem: isol fund sys}
    Let $b \in Z$. Then $\tp(b/\dcl(A,X))$ is isolated.   
\end{lemma}

\begin{proof}
    Let $p(y) = \tp(b/\dcl(A,X))$. 
    
    As $M \prec \bar{M}$, there is some $b_0 \in Z(M)$ (assuming $Z$ is nonempty). Note that $X(M)$ is a normal subset of $M$ so $M$ is atomic over $\dcl(A,X(M))$. Hence $\tp(b_0/\dcl(A,X(M)))$ is isolated by a formula $\psi(y,\bar{d}_0)$ where $\bar{d}_0$ is a tuple from $\dcl(A,X(M))$. Let the realizations of this type be $Q_{\bar{d}_0}$. 
    
    Let $c \in X_1$ with $f(b_0,c) = b$. Then $$  \exists y_0 \in Z \  \psi(y_0,\bar{d}_0) \wedge f(y_0,c)= y_1$$ defines a nonempty definable subset of $Z$. Note that if $c \in \dcl(A,X(M))$, then this formula is $\dcl(A,X(M))$-definable and is therefore realized in $M$. 
    
    It remains to show that this formula isolates $p(y)$. The formula is clearly over $\dcl(A,X)$ and $b$ satisfies it, so it is in the type. Now let $b'$ realize the above formula. Let $b_0'$ witness that $$\psi(b_0',d_0) \wedge f(b_0',c) = b'.$$ Let $\sigma \in \aut(\bar{M}/\dcl(A,X))$ take $b_0$ to $b_0'$. Then $\sigma(b) = f(\sigma(b_0),c) = f(b_0',c) = b'$. It then follows that $b'$ has the same type as $b$ over $\dcl(A,X(M))$. So the type $p(y)$ is isolated over $\dcl(A,X(M))$ with tuple of parameters $\bar{d}_0,c$.
\end{proof}

\begin{remark}\label{rmk: fund in M}
    First note that in the above proof $c \in \dcl(A,X(M))$ iff $b \in M$. 
    
    Secondly, by the normality of $\dcl(A,X(M))$, for any $\sigma \in \aut(\bar{M}/A)$, then $\psi(y,\sigma(\bar{d}_0))$ isolates a complete type over $\dcl(A,X(\sigma(M)))$. Then the formula $\exists y_0 \in Z \  \psi(y_0,\sigma(\bar{d}_0)) \wedge f(y_0,\sigma(c))= y_1$ isolates the type of some $b \in Z$ over $\dcl(A,X)$.    
\end{remark}

\begin{definition}
We define two $A$-definable binary relations on $Z^2$ as follows. \begin{enumerate}
    \item Say $(b_1,b_1')$ is $E_{\text{int}}$-equivalent to $(b_2,b_2')$ if there is a $c \in X_1$ such that $f(b_1,c) = b_2$ and $f(b_1',c) = b_2'$. The int stands for intrinsic. 
    \item Say $(b_1,b_1')$ is $E_{\text{ext}}$-equivalent to $(b_2,b_2')$ if there is a $c \in X_1$ such that $f(b_1,c) = b_1'$ and $f(b_2,c) = b_2'$. The ext stands for extrinsic. 
\end{enumerate} 
\end{definition}

\begin{proposition}\label{prop: binding groupoid}
    Let $b \in Z$. By the preceding lemma $\tp(b/\dcl(A,X))$ is isolated. Let $d$ be a tuple from $\dcl(A,X)$ isolating the type $\tp(b/\dcl(A,X))$. Let $Q_{d}$ be the definable set of realizations of the type. 
 
    \begin{enumerate}
        \item Then $E_{\text{int}}$ and $E_{\text{ext}}$ restricted to $Q_{d}^2$ are definable equivalence relations. 
        \item $H_{d}^+ := (Q_{d})^2/E_{\text{int}}$ is a $\dcl(A,d)$-definable group with a $A$-definable action on $Q_{d}$ abstractly isomorphic to the action of $\aut(Q_{d}/\dcl(A,X))$ on $Q_{d}$ via elementary permutation.
        \item A descent argument shows that each $H^+_{d}$ is in definable bijection with a group $H^+$ defined over $A$. Moreover $H^+$ acts on $Z$ on the left abstractly isomorphically to the action of $\aut(Z/\dcl(A,X))$ on $Z$. 
        \item Define $O:=Z/H^+ \subseteq \dcl(A,X)$. Let $\pi:Z \to O$ be the $A$-definable quotient map. For any $b \in Z$, $\tp(b/\dcl(A,X))$ is isolated over $\pi(b) = d \in \dcl(A,X)$ by the formula $\pi(y) = d$.
        \item For any $Q_{d}$, $Q_{d'}$, define $H_{dd'} \subseteq X_1$ to be $c \in X_1$ such that $f(b,c) \in Q_{\bar{d}'}$ for any $b \in Q_{d}$. This set is $\dcl(A,d,d')$-definable. Write $H_{dd}$ as $H_{d}$. 
        \item $H_{d} \cong Q_{d}^2/E_{\text{ext}}$. Moreover $H_{d}$ is an $\dcl(A,d)$ definable group acting via $f$ on the right on $Q_{d}$ giving $(H^+,Q_{d},H_{d})$ the structure of a $\dcl(A,d)$-definable biprincipal homogeneous space.
        \item If $d \in O(A)$, then there is a $b \in Q_{\bar{d}}(M)$ and $\dcl(A,b)$ generates an $H_d$-definable Galois extension of $A$ in $M$. Otherwise in general an element $b \in Q_{\bar{d}}(M)$ generates an $H_d$-definable Galois extension of $\dcl(A,d)$ in $M$.
    \end{enumerate}
\end{proposition}

\begin{proof}  More or less as in \cite{EhudHrushovski2002} appendix $B$.
    \begin{enumerate}
        \item For reflexivity of $E_{\text{int}}$ on $Q_{\bar{d}}^2$: Let $b_1,b_1' \in Q_{\bar{d}}$. There exists a $c \in X_1$ such that $f(\bar{b}_1,c) = \bar{b}_1$. We need to show that this same $c$ satisfies $f(\bar{b}_1',c) = \bar{b}_1'$. Let $\sigma \in \aut(\bar{M}/\dcl(A,X))$ with $\sigma(\bar{b}_1) = \bar{b}_1'$. Apply $\sigma$ to $f(\bar{b}_1,c) = \bar{b}_1$.
        
        For symmetry of $E_{\text{int}}$ on $Q_{\bar{d}}^2$ let $(b_1,b_1')$ be $E_{\text{int}}$-related to $(b_2,b_2')$ via $c \in X_1$. There exists $c' \in X_1$ such that $f(b_2,c') = b_1$. We want to show $f(b_2',c') = b_1'$. Again, let $\sigma \in \aut(\bar{M}/\dcl(A,X))$ with $\sigma(\bar{b}_1) = \bar{b}_1'$. As $f(b_1,c) = b_2$ and $f(b_1',c) = b_2'$, $\sigma(b_2) = \sigma(f(b_1,c)) = f(\sigma(b_1),c) = f(b_1',c) = b_2'$. Then applying $\sigma$ to $f(b_2,c') = b_1$, we see $f(b_2',c') = \sigma(f(b_2,c')) = \sigma(b_1) = b_1'$ as desired. Transitivity follows similarly. 
        
        That $E_{\text{ext}}$ is an equivalence relation on $Q_{\bar{d}}^2$ is immediate.

        \item The set $Q_{\bar{d}}$ is $\dcl(A,\bar{d})$ definable and $E_{\text{int}}$ is $A$-definable so the quotient $H^+_d$ is $\dcl(A,\bar{d})$-definable by EI. Define a map $\iota_{d}:\aut(Q_{d}/\dcl(A,X)) \to H^+_{d}$ by sending $\sigma \mapsto [(\sigma(b),b)]$ for some $b \in Q_{d}$. This is well defined independent of the choice of $b$: If $b_1,b_2 \in Q_{d}$, there is $c \in X_1$ with $f(b_1,c) = b_2$, then also $f(\sigma(b_1),c) = \sigma(b_2)$. So $[(\sigma(b_1),b_1)] = [(\sigma(b_2),b_2)]$. It is surjective by strong homogeneity of the monster model over stably embedded sets. 
        
        It is injective: Let $\sigma_1,\sigma_2 \in \aut(Q_{d}/\dcl(A,X))$ with the associated classes $[(\sigma_1(b),b)] = [(\sigma_2(b),b)]$. This means that there exists a $c \in X_1$ such that $f(b,c) = b$ and $f(\sigma_1(b),c) = \sigma_2(c)$. Then $b = f(b,c) = \sigma_1^{-1}\sigma_2(b)$. So $\sigma_1(b) = \sigma_2(b)$. Since an elementary permutation $\sigma$ of $Q_{\bar{d}}$ fixing $\dcl(A,X)$ is uniquely determined by it's action on any $b \in Q_{d}$, we have $\sigma_1 = \sigma_2$. 

        Let $\pi_{int}:Q^2_d \to H_d^+$ be the definable quotient map given by EI. The bijection $\iota_{d}:\aut(Q_{d}/\dcl(A,X)) \to H^+_{d}$ induces a group structure on $H^+_d$. It is an exercise to check that the group structure is defined by the following formula in the variables $x_1,x_2,x_3$ ranging over $H_d^+$:

        $\exists y_1,y_2,y_3 \in Q_d \ \pi_{int}(y_2,y_1) = x_2 \wedge \pi_{int}(y_3,y_2) = x_1 \wedge \pi_{int}(y_3,y_1) = x_3.$
        
        The action of $H^+_{d}$ on $Q_{d}$ is $A$-definable and agrees with the action of $\aut(Q_{\bar{d}}/\dcl(A,X))$: Let $[(b',b)] \in H^{+}_{\bar{d}}$ and $b'' \in Q_{d}$. Then we evaluate $[(b',b)]b''$ as the unique $b'''$ such that $\pi(b''',b'') = [(b',b)]$. This is $A$-definable. Let $x$ and $y,y_1$ be variables ranging over $H^{+}_{\bar{d}}$ and $Q_{d}$ respectively, then the action $xy = y_1$ is $A$-definable by the formula $\pi_{int}(y_1,y)=x$.   It is clear that this agrees with the action of $\aut(Q_{\bar{d}}/\dcl(A,X))$ on the left.
        
        \item We adapt the argument showing $G$ is $F$-definable from \cite{EhudHrushovski2002} Theorem B.1. Let $H^+_{d_1}$ and $H^+_{d_2}$ be two such groups. Note that $\aut(Q_{d_1}/\dcl(A,X))$ and also $\aut(Q_{d_2}/\dcl(A,X))$ are both canonically isomorphic to $\aut(Z/\dcl(A,X))$ as any such elementary permutation of $Z$, $Q_{d_1}$, or $Q_{d_2}$ is entirely determined by the image of any single fundamental system $b$. So making the necessary identifications, we have a map $\iota_{d_2}\circ \iota_{d_1}^{-1}:H^+_{d_1} \to H^+_{d_2}$. It is easy to verify that this map sends $[(b_1,b_1')] \in H^+_{d_1}$ to $[(b_2,b_2')] \in H^+_{d_2}$ such that $\exists c \in X_1$ with $f(b_1,c) = b_2$ and $f(b_1',c) = b_2'$. Then $\iota_{d_2}\circ \iota_{d_1}^{-1}(x_1) = x_2$ is definable over $d_1d_2$ via $\exists y_1,y_2 \in Q_{d_1}, \exists y_3,y_4 \in Q_{d_2}, \ \exists c \in X_1, \  \pi_{\text{int}}(y_2,y_1) = x_1 \wedge \pi_{\text{int}}(y_4,y_3) = x_2 \wedge f(y_2,c) = y_4 \wedge f(y_1,c) = y_3$.

        Let $\phi(x)$ isolate $\tp(d_0/A)$ as in the proof of Lemma \ref{lem: isol fund sys} and $D$ be the set of realizations. Note that each $d$ is a concatenation of $d_0$ and an element $c \in X_1$. Consider the $A$-definable set $x_0 \in D  \wedge  z \in X_1 \wedge x \in H_{x_0z}^+$ i.e. the parameterized family of $H^+_d$. Define a relation via: $(x_0,z,x) \equiv (x_0',z',x')$ if $\iota_{x_0'z'}\circ \iota_{x_0z}^{-1}(x) = x'$. It is clear that this identifies distinct elements of the different $H_d^+$ which correspond to the same elementary permutation of $\aut(Z/\dcl(A,X))$ from which we see that it is an $A$-definable equivalence relation. In fact it is the equivalence relation whose classes are fibers of the piecewise map $\sqcup_d\iota_d^{-1}:\sqcup H_{d}^{+} \to \aut(Z/\dcl(A,X))$. Let $H^+$ be the quotient. It is clear that $H^+$ acts on $Z$ on the right $A$-definably isomorphically to the action of $\aut(Z/\dcl(A,X))$ on $Z$.

        \item That $O$ is contained in $\dcl(A,X)$ is because the action of $H^+$ is the action of $\aut(Z/\dcl(A,X))$ on $Z$. The orbits (equivalently, the fibers of $\pi:Z \to O$) are exactly the sets $Q_d$.

        \item This is a definition. 

        \item Define $h_{\text{ext}}:Q_d^2 \to H_d$ via $(b,b') \mapsto c$ such that $f(b,c) = b'$. For $b \in Q_d$, define $\rho_b:\aut(Q_d/\dcl(A,X)) \to H_d$ via $\sigma \mapsto c_\sigma \in H_d$ such that $f(b,c_\sigma) = \sigma(b)$. It is clear that $\rho_b$ is a bijection and that the identity automorphism gets sent to an element in $\dcl(A,d)$. The formula $\exists y \in Q_d f(f(y,x_1),x_2) = f(y,x_3)$ defines the induced group operation on $H_d$ and the action of $H_d$ on the right on $Q_d$ is given by $f$. Let $\sigma \in \aut(Q_d/\dcl(A,X))$, $c \in H_d$, and $b \in Q_d$. The equality $\sigma(f(b,c)) = f(\sigma(f),c)$ shows the compatibility of the biprincipal homogeneous space as the action of $H^+$ is via elementary permutation.

        \item By Theorem \ref{thm: cohomological fact} part 1, as $(Q_d,H_d)$ is a $\dcl(A,d)$ definable torsor, for $b \in Q_d$, $\dcl(A,b)$ is a definable Galois extension of $\dcl(A,d)$ if  $H_d(\dcl(A,b)) = H_d(\dcl(A,d))$. But this follows: Let $g(b) = c \in H_d(\dcl(A,b))$. Let $\sigma \in \aut(\bar{M}/\dcl(A,d))$, then $\sigma$ fixes $\tp(b/\dcl(A,X))$ which we called $Q_d$. As $g(y) = \sigma(c) \in \tp(b/\dcl(A,X))$ so $\sigma(c) = c$, so $c \in \dcl(A,d)$.
        
        Note that if $d \in O(A)$ then $\exists b \in Q_d(M)$ (See Remark \ref{rmk: fund in M}) and then $\dcl(A,b)$ is a definable Galois extension of $A$.
        \end{enumerate}
\end{proof}

\section{Multiplicity of definable Galois extensions for internality data}\label{sec: mult}

In this section we prove an analogue of the following result in the totally transcendental setting.

\begin{theorem}\cite{DeligneMilne2022} Let $Y$ be an ordinary homogeneous linear differential equation over a differential field $K$. Assume that there is at least one nontrivial Picard-Vessiot extension $L/K$ for $K$ in $K^{\diff}$, a differential closure of $K$. Let $H(C_{K^{\diff}})$, a linear algebraic group defined over $C_K$, be the (extrinsic) Picard-Vessiot differential Galois group of the extension $L/K$. Then the usual algebraic Galois cohomology $$\HH^1_{\alg}(C_{K^{\diff}}/C_K,H(C_K))$$
classifies (counts) exactly the set of distinct Picard-Vessiot extensions for $Y$ in $K^{\diff}$ with the base point of the pointed cohomology set corresponding to $L/K$.
\end{theorem}

Let $(Y,X)$ be fixed internality data. Construct $X_1$ and $Z$ as in the construction of the binding groupoid. We now focus on the case where $X_1$ is a $A$-definable group and $Z$ is a $A$-definable right PHS for $X_1$.

\begin{theorem}
    Let $Z$ be a right $A$-definable PHS for an $A$-definable group $X_1$. 
    \begin{enumerate}
        \item $E_{\text{int}}$ and $E_{\text{ext}}$ are then $A$-definable equivalence relations on $Z^2$. 
        \item $Z^2/E_{\text{int}}$ is an $A$-definable group which acts on the left on $Z$ as the set of all permutations commuting with the right action of $X_1$. It has $H^+$ as the $A$-definable subgroup of elementary permutations commuting with the right action of $X_1$.
        
        \item Assume there exists at least one $b_0 \in Z(M)$ with $\pi(b_0) = d_0 \in O(A)$ and extrinsic group $H_{d_0}$. Then the set of definable Galois  extensions in $M$ generated by elements of $Z(M)$ embeds in $\HH^{1}_{\defin}(M/A,H_{d_0}(M))$. If the cohomology $\HH^{1}_{\defin}(M/A,X_1(M))$ is trivial then this embedding is an isomorphism. This extends Proposition 3.2.19 of \cite{MeretzkyThesis}.
    \end{enumerate}
\end{theorem}

\begin{proof}
    \begin{enumerate}
        \item Reflexivity of $E_{\text{int}}$: Let $e \in X_1$ be the identity. Then $f(b_1,e) = b_1$ and $f(b_1',e) = b_1'$ for any $(b_1,b_1') \in Z^2$. Symmetry of $E_{\text{int}}$: Let $(b_1,b_1')$ be $E_{\text{int}}$-equivalent to $(b_2,b_2')$ witnessed by $c \in X_1$. Then $c^{-1} \in X_1$ witnesses that $(b_2,b_2')$ is $E_{\text{int}}$-equivalent to $(b_1,b_1')$. Transitivity of $E_{\text{int}}$: Let $(b_1,b_1')$ be $E_{\text{int}}$-equivalent to $(b_2,b_2')$ witnessed by $c_{12} \in X_1$. Let $(b_2,b_2')$ be $E_{\text{int}}$-equivalent to $(b_3,b_3')$ witnessed by $c_{23} \in X_1$. Then the product in $X_1$ of $c_{12}c_{23}$ witnesses that $(b_1,b_1')$ is $E_{\text{int}}$-equivalent to $(b_3,b_3')$.

        Reflexivity, symmetry, and transitivity of $E_{ext}$ are all immediate using the fact that $Z$ is a PHS for $X_1$.

        \item The definition of $E_{int}$ and an easy computation show that $Z^2/E_{int}$ determines uniquely a permutation of $Z$ commuting with the right action of $X_1$. It is clear that $H^+$ is just those permutations preserving type over $\dcl(A,X)$ and is defined by exactly the formula in the variable $x\in Z^2/E_{int}$ given as $\exists y_1,y_2 \in Z \  \exists d \in D \times X_1\ \pi_{int}(y_1,y_2) = x \wedge y_1 \equiv_d y_2$ where $y_1 \equiv_d y_2$ says $y_1$ and $y_2$ both satisfy same type over $d$, i.e. the both satisfy the formula given in Remark \ref{rmk: fund in M}. Recall that $D$ is defined in the proof of part 3) of Proposition \ref{prop: binding groupoid}.

        The induced group operation on $Z^2/E_{int}$ is $A$-definable: The product of two $E_{\text{int}}$-classes $[(b_1,b_1')]$ and $[(b_2,b_2')]$ is computed as follows. As $Z$ is a PHS for $X_1$ there is a unique $c \in X_1$ such that $f(b_2,c) = b_1'$. Then $[(b_2,b_2')] = [(b_1',f(b_2',c))]$. So $[(b_1,b_1')][(b_2,b_2')] = [(b_1,b_1')][(b_1',f(b_2',c))]$ we then define $[(b_1,b_1')][(b_1',f(b_2',c))] := [(b_1,f(b_2',c))]$.

        The action of $Z^2/E_{int}$ or $H^+$ on $Z$ is defined as before via $[(b,b')]b' := b$.

        %Let $[(b,b')]$ be an element of $H^+$. 
        
        %If $b'$ has the same type as $b$ over $\dcl(A,X)$, then by the stable embeddedness of $X$, there is a $\sigma \in \aut(Z/\dcl(A,X))$ such that $\sigma(b') = b$ in which case $[(b,b')]b' = \sigma(b') = b$ and the actions agree.
        
\item We adapt Proposition 3.2.19 of \cite{MeretzkyThesis}. 
    
Let $b_1 \in Z(M)$, such that $\dcl(A,b_1)$ is a definable Galois extension for the equation. By the last part of Proposition \ref{prop: binding groupoid}, $f(b_1) = d_1 \in O(A)$ and therefore $H_{d_1,d_0}(M)$ is an $A$-definable right PHS for $H_{d_0}(M)$ which gives the desired class. Recall that $H_{d_1,d_0}(M)$ is defined as $$\{c \in X_1(M) \ : \   
\exists b_1 \in Q_{d_1}(M), \  \exists b_0 \in Q_{d_0}(M), \  f(b_1,c) = b_0\}$$ 
It is clear that if $b_1,b_2 \in Z(M)$ generate the same definable Galois extension then there is a $c \in X_1(A)$ with $f(b_1,c) = b_2$. This yields an $A$-definable bijection between the right PHSs $H_{d_1,d_0}(M)$ and $H_{d_2,d_0}(M)$, and so the resulting cohomology classes are isomorphic, giving injectivity.

Note also that $H_{d_0,d_1}(M)$ is an $A$-definable left PHS for $H_{d_0}(M)$ so we could phrase everything in terms of left PHSs/cohomology if we wanted to.

For the moreover clause we need the long exact sequence in cohomology. Let $P(M)$ be a right $A$-definable PHS for $H_{d_0}(M)$. By the assumption that $\HH^1_{\alg}(M/A,X_1(M))$ is trivial, the class associated to $P(M)$ must be in the image of the connecting map $\delta^1$:
    \begin{align*}
        1  \to H_{d_0}(A) &\to X_1(A) \to (X_1/H_{d_0})(A) \xrightarrow{\delta^1} \\
        & \HH^1_{\alg}(M/A,H_{d_0}(M)) \to 1 \  (= \HH^1_{\alg}(M/A,X_1(M)))
    \end{align*}

    To see this, note first that the action of $H_{d_0}$ gives an $A$-definable equivalence relation on the $A$-definable group $X_1$. So by EI, the quotient is an $A$-definable set of left cosets $(X_1/H_{d_0})$. Note also the quotient map $\pi$ is a $A$-definable function. Given an $A$-point, $p \in (X_1/H_{d_0})(A)$, we have that $\dcl(A,p) = A$. The preimage $\pi^{-1}(p)$ is then an $A$-definable left coset of $H_{d_0}(M)$ in $X_1(M)$ which is then a right $A$-definable PHS for 
    $H_{d_0}(M)$. 
    
    So the assumption requiring the vanishing of the cohomology together with the long exact sequence implies that any definable right PHS for $H_{d}(M)$ must be a $A$-definable coset of $H_{d}(M)$ in $X_1(M)$. But this is exactly something of the form $H_{d_1,d_0}(M)$ because $Z(M)$ is a right PHS for $X_1(M)$: Take an element $c$ in this coset. Apply the inverse $c^{-1}$ to $b_0$ to get a new solution $b_1$ with $f(b_1) = d_1$. Then the coset is actually equal to $H_{d_1d_0}(M)$ as they intersect nontrivially. In the previous step we use that there is already one proper definable Galois extension, so that we may take such an element $b_0$.
    \end{enumerate}
\end{proof}

\begin{remark}\label{rmk: assmp ii}
In general there seems to be no reason why $X_1$ needs to carry a group structure making $(Z,X_1)$ into an $A$-definable PHS with action $f$. However this assumption is satisfied whenever the internality data we begin with $(Y,X)$ has the form of an $A$-definable PHS. 

Very often the internality data given is in the form of an equation on a group. By this, concretely we mean an equation coming from a $G$-primitive in the model-theoretic sense of Pillay-Sokolovic or $V$-primitive in the model-theoretic sense defined in \cite{MeretzkyThesis} as in Definition \ref{def: prim elt} above. Indeed, starting with a fixed extension, Theorem \ref{thm: cohomological fact} \cite{OmarDavid2026} tells us exactly that we can find a torsor which gives rise to the extension i.e. the extension has a generator which is a primitive element in this sense.  So in practice, this assumption is very often satisfied. In the differential algebraic setting of $T=\DCF_{0,m}$, such equations are the logarithmic differential equations of Kolchin on a (differential) algebraic groups \cite{OmarDavid2026}. %For a concrete example, in the Picard-Vessiot setting, where the internality data is exactly an homogeneous linear ordinary diferential equation, the set of fundamental systems $Z$ is right PHS for all of $GL_n(M) = X_1(M)$.
\end{remark}%(Aside: What if anything can be said when the internality data is assumed to be related to the covering map context of Scanlon?) 

\section{Colimit formulas and boundedness for differential fields}\label{sec: colim diff}

We begin by reviewing the basic relationships between algebraic algebraic Galois cohomology and Galois extensions. We will then survey, prove, and propose some analogous statements in the differential algebraic setting.  

\begin{fact}\label{fact: alg coh triv} \cite{Serre1979}
    Let $F$ be a perfect field. The following are equivalent:
    
    \begin{enumerate}
        \item $F = F^{\alg}$.
        \item For any finite algebraic group $G$ defined over $F$, $$\HH^1_{\alg}(F^{\alg}/F,G(F^{\alg})) = 1.$$
        \item For any algebraic group $G$ defined over $F$, $$\HH^1_{\alg}(F^{\alg}/F,G(F^{\alg})) = 1.$$
    \end{enumerate}
\end{fact}

We next recall the usual colimit formula for Galois cohomology. 

\begin{fact}\label{fact: alg colim}\cite{Serre1979}
    Let $F$ be a perfect field and $G$ be an algebraic group defined over $F$. Then $$\HH^1_{\alg}(F^{\alg}/F,G(F^{\alg})) = \varinjlim \HH^1_{\alg}(E/F,G(E))$$ where the limit is taken over the directed system of finite Galois extensions $E$ of $F$.
\end{fact}

We finally recall the notion of boundedness of a field and some cohomological consequences.

\begin{fact}\cite{Serre1979}
    A perfect field $F$ is said to be bounded if any of the following equivalent conditions hold

    \begin{enumerate}
        \item $F$ has finitely many Galois extensions of each finite degree. 
        \item For any linear algebraic group $G$ over $F$, $\HH^1_{\alg}(F^{\alg}/F,G(F^{\alg}))$ is finite.
        \item For any linear algebraic group $G$ over $F$, there is a finite Galois extension $E$ of $F$ such that  $\HH^1_{\alg}(E/F,G(E))\cong \HH^1_{\alg}(F^{\alg}/F,G(F^{\alg}))$.
    \end{enumerate}
\end{fact}

We now put together some results in the literature which give differential algebraic analogues of the above, and propose notions of boundedness suitable for a differential field.

In the rest of this section $K$ is a differential field of characteristic $0$ with algebraically closed field of constants, $C_K = (C_K)^{\alg}$. We use the conventions of \cite{meretzky2026galoistheoryautomorphismgroups} for finite and infinite Picard-Vessiot (PV) extensions. Namely a PV extension $L$ of $K$ is an extension generated by fundamental systems (in $K^{\diff}$) of solutions to ordinary homogeneous linear differential equations over $K$. If $L$ is generated over $K$ as a differential field by the fundamental system of a single equation (it is equivalent to ask for finitely many) we will call $L$ a PV extension of finite type. 

We recall the definition of the sequence of Picard-Vessiot envelopes of a differential field $K$ and the Picard-Vessiot closure from \cite{meretzky2026galoistheoryautomorphismgroups} for example. Let $K^{PV_1}$ be the union of all Picard-Vessiot extensions of $K$. (It is enough to ask for them inside of $K^{\diff}$.) As is well known, see Example 3.33 of \cite{magid1994lectures},  $K^{PV_1}$ may itself not be closed in $K^{\diff}$, and the union of all PV extensions of $K^{PV_1}$ in $K^{\diff}$ we write as $K^{PV_2}$ and so forth. We define $K^{PV_\infty}=\cup_n K^{PV_n}$ which is itself $PV$ closed and is called the Picard Vessiot closure of $K$. As in  \cite{meretzky2026galoistheoryautomorphismgroups}, this yields a chain of normal differential subfields of the minimal closure of $K$, $$K \subseteq K^{PV_1} \subseteq K^{PV_2} \subseteq ... K^{PV_\infty}\subseteq \mcl(K)\subseteq K^{\diff}.$$

There are then at least three intermediate differential fields between $K$ and $K^{\diff}$ (or intermediate sets between $A$ and $M$) which each, in different senses, generalize the algebraic closure. Namely, $K^{PV_\infty}\subseteq \mcl(K)\subseteq K^{\diff}$. 

The following differential algebraic version of Fact \ref{fact: alg coh triv} was shown in both \cite{PILLAY2017809} and \cite{Minchenko_2019}: In fact, \cite{PILLAY2017809}  covers the generalized strongly normal theory, while \cite{Minchenko_2019} works in multiple commuting derivations in the linear case with an additional assumption.  

\begin{theorem}\label{thm: pv closed}\cite{PILLAY2017809}
    Let $K$ be a differential field. The following are equivalent:
    
    \begin{enumerate}
        \item $K$ is algebraically closed and PV-closed. 
        \item For any $G$ a linear algebraic group in the constants then $$H^1_{\delta}(K^{\diff}/K,G(K^{\diff})) = 1$$
        \item For any $G$ be a linear differential algebraic group over $K$ we have $$H^1_{\delta}(K^{\diff}/K,G(K^{\diff})) = 1$$
    \end{enumerate}
\end{theorem}

\begin{remark}
    If $K = K^{\diff}$, then from the definition,  $H^1_{\delta}(K^{\diff}/K,G(K^{\diff})) = 1$ for any linear differential algebraic group $G$ over $K$. We show below in Proposition 9.9 (in the more general model theoretic context) that if $K=\mcl(K)$ then again $H^1_{\delta}(K^{\diff}/K,G(K^{\diff})) = 1$ for any linear differential algebraic group $G$ over $K$. However, the above Theorem \ref{thm: pv closed} together with abundant examples where the inclusions $K^{PV_\infty}\subseteq \mcl(K)\subseteq K^{\diff}$ are strict show that the converses are false.
\end{remark}

\begin{corollary}\label{cor: pv closed}\cite{AnandDavidComm}
    Let $K$ be a differential field and $G$ a linear algebraic group in the constants of $K$. Then $$H^1_{\delta}(K^{PV_\infty}/K,G(K^{PV_\infty})) \cong H^1_{\delta}(K^{\diff}/K,G(K^{\diff})).$$
\end{corollary}

\begin{proof}
    Immediate from Theorem \ref{thm: SES} applied to $K \subseteq K^{PV_\infty} \subseteq K^{\diff}$ and Theorem \ref{thm: pv closed}.
\end{proof}

In preparation to show some differential algebraic versions of the colimit formula of Fact \ref{fact: alg colim}, we make some remarks about various kinds of subfields of $K^{PV_\infty}$. We will comment on finitely differentially generated subfields, normal subfields, Picard-Vessiot extensions, and iterated Picard-Vessiot extensions. 

We contine to follow \cite{meretzky2026galoistheoryautomorphismgroups} for notation and cite results from \cite{magid2022completepicardvessiotclosure}. 

An iterated PV (IPV) extension of $K$ is a differential field extension $K \subseteq L\subseteq K^{PV_\infty}$, such that there exists an infinite sequence of proper Picard-Vessiot extensions $L_i/L_{i-1}$, 
not necessarily of finite type, such that $L$ is the union of that chain, or there exists a finite such chain $K = L_0 \subseteq L_1 \subseteq L_2 \subseteq \cdots L_n$ with $L_n = L$.  

Theorem 2 of \cite{magid2022completepicardvessiotclosure} says that $K \subseteq L \subseteq K^{PV_{\infty}}$ is finitely differentially generated if and only if $L$ is a differential subfield of an IPV extension which is the union of a finite chain where each $L_i/L_{i-1}$ is a PV extension of finite type. We will call IVP extensions of this form IVP extensions of finite type. 

Proposition 3.2 of \cite{meretzky2026galoistheoryautomorphismgroups}, shows that any normal extension, $K \subseteq L \subseteq K^{PV_\infty}$, is an IPV extension. 

Lemma 2 of \cite{magid2022completepicardvessiotclosure} says that the compositum of two (or finitely many) IPV  extensions (in our sense) is again an IPV extension. The proof shows that if the two extensions are of finite type the compositum will be of finite type.  So normal IPV extensions form a directed system whose union is $K^{PV_\infty}$. This statement is tautological because $K^{PV_\infty}$ is itself a normal IPV extension. We address this as follows:

Firstly, restricting to finitely differentially generated normal IPV extensions yields exactly the family of PV extensions of finite type: 

\begin{lemma}\label{lem: normal IPV is PV}
    Let $K\subseteq L\subseteq K^{PV_\infty}$. The following are equivalent: 
    \begin{enumerate}
        \item $L$ is normal and finitely differentially generated over $K$.
        \item $L$ is a normal IPV extension of finite type over $K$.
        \item $L$ is a PV extension of finite type over $K$.
    \end{enumerate} 
\end{lemma}

\begin{proof}
    We show (1) implies (2) implies (3). The remaining (3) implies (1) is immediate.
    
    Let $L$ be normal and finitely differentially generated over $K$. As we assumed $C_K=C_K^{\alg}$, $L$ is then a strongly normal extension of $K$. 

    As $L$ is finitely differentially generated, Theorem 2 of \cite{magid2022completepicardvessiotclosure} says that $L$ is a differential subfield of an IPV extension of finite type. In the proof of Theorem 2 of \cite{magid2022completepicardvessiotclosure}, an additional assumption of normality yields that $L$ is exactly an IPV extension of $K$ of finite type.
    %As $L$ is finitely differentially generated and then by Proposition 3.2 of \cite{meretzky2026galoistheoryautomorphismgroups}, $L$ must be a subfield of an IVP extension of finite type. 
    But then the Galois group of $L$ must have a composition series of linear algebraic groups and so must be linear. 
    
    By Fact 1.20 of \cite{AnandDavidComm}, a strongly normal extension with a linear algebraic Galois group must be a PV extension of finite type.

\end{proof}

\begin{proposition}
    Let $G$ be a linear differential algebraic group defined over $K$. Then $$\varinjlim H^1_{\delta}(L/K,G(L)) \cong H^1_{\delta}(K^{PV_1}/K,G(K^{PV_1}))$$ where the limit is again taken over the family of normal, finitely differentially generated extensions $K \subseteq L\subseteq K^{PV_\infty}$ or equivalently, the family of Picard-Vessiot extensions of finite type. 
\end{proposition}

\begin{proof}
    The normal, finitely differentially generated extensions $K \subseteq L\subseteq K^{PV_\infty}$ form a directed system under inclusion. By the previous Lemma \ref{lem: normal IPV is PV}, these are exactly the PV-extensions of finite type, so their direct limit is exactly $K^{PV_1}$. We then apply Lemma \ref{lem: colim}.
\end{proof}

Lemma \ref{lem: normal IPV is PV} above, verifies the remark (immediately following Definition 4) of \cite{magid2022completepicardvessiotclosure}, that an IPV extension of finite type which is not a PV extension of finite type cannot be normal. Therefore, to obtain a colimit formula for $K^{PV_{\infty}}$ this we must work with the set of normal closures of IPV extensions of $K$ of finite type:

\begin{proposition}\label{prop: colim ipv}
      Let $G$ be a linear differential algebraic group $G$ over $K$. Then $$\varinjlim H^1_{\delta}(L/K,G(L)) \cong H^1_{\delta}(K^{PV_\infty}/K,G(K^{PV_\infty}))$$ where the limit is taken over the family of normal closures of IPV extensions of finite type. 
\end{proposition}

\begin{proof}
    By Proposition 3.2 of \cite{meretzky2026galoistheoryautomorphismgroups}, the normal closure of an IPV extension of finite type is again an IPV extension (although not necessarily of finite type). Then by Lemma 2 of \cite{magid2022completepicardvessiotclosure}, the normal closures of IPV extensions of finite type form a directed system whose direct limit is $K^{PV_\infty}$. We then apply Lemma \ref{lem: colim}.
\end{proof}

From Corollary \ref{cor: pv closed} and from Proposition \ref{prop: colim ipv} we obtain the following analogue of Fact \ref{fact: alg colim}, the usual colimit formula in the differential algebraic setting:

\begin{proposition}
    Let $G$ be a linear differential algebraic group defined over a differential field $K$. Then $$\varinjlim\HH^1_{\delta}(L/K,G(L)) \cong \HH^1_{\delta}(K^{\diff}/K,G(K^{\diff}))$$ where the limit is taken over the directed system of normal closures of iterated PV extensions of $K$ of finite type.
\end{proposition}

We then propose the following definition of boundedness for a differential field which can be added to the list of such notions in Section 2.3 of \cite{AnandDavidComm}:

\begin{definition}\label{def: diff bdd}
    A differential field $K$ is said to be bounded with respect to a differential algebraic group $G$ defined over $K$ if there is an iterated PV extension of finite type, whose normal closure $L$, satisfies $$\HH^1_{\delta}(L/K,G(L)) \cong \HH^1_{\delta}(K^{\diff}/K,G(K^{\diff})).$$
\end{definition}

\begin{remark}
    Note that Magid's example in Section 4 of \cite{magid2022completepicardvessiotclosure} of an normal IPV extension $E$, is remarked to be exactly the normal closure of an IPV extension of finite type, $E_0$. 
\end{remark}

\section{The minimal and definable Galois closures and colimit formulas for definable Galois cohomology}\label{sec: tt setting}

Motivated by the above observations in the differential algebraic setting, we give some analogous results in the model theoretic setting. Let $T$ be a t.t. theory with EI. Let $A$ be a small set of parameters and $M$ a prime model over $A$. Let $X$ be an $A$-definable set.

We make the following definitions and adjustments to the terminology to be consistent with the previous section.  

\begin{definition}
    Let $A$ be a $\dcl$-closed set of parameters. Let $X$ be an $A$-definable set.  
    \begin{enumerate}
        \item In this section, a definable Galois extension over $A$ relative to $X$ as defined in Definition \ref{def: def gal ext} will be called a definable Galois extension of finite type. 
        \item By a definable Galois extension $A \subseteq B \subseteq M$ with respect to $X$ we mean a union of definable Galois extensions of finite type over $A$ relative to $X$. 
        \item By $A^{X\text{DG}_1}$ we mean the union of all $X$-definable Galois extensions $A\subseteq B$ in $M$.  
        \item Having defined $A^{X\text{DG}_n}$ we define $A^{X\text{DG}_{n+1}}$ to be the union of all $X$-definable Galois extensions $A^{X\text{DG}_n}\subseteq B$ in $M$.  
        \item By $A^{X\text{DG}_\infty}$ we mean $\cup_n A^{X\text{DG}_n}$.
        \item An iterated definable Galois extension $A \subseteq B \subseteq M$ relative to $X$ is a $\dcl$-closed subset such that either there exists an infinite sequence of $\dcl$-closed subsets $A=B_0 \subseteq B_1 \subseteq \cdots $ with each $B_i$ a proper definable Galois extension of $B_{i-1}$ such that $B= \cup_iB_i$, or there exists a finite such sequence $A=B_0 \subseteq B_1 \subseteq \cdots B_n=B$. If $B$ is the union of a finite such chain where each $B_i/B_{i-1}$ is a definable Galois extension of finite type, we call $B$ an iterated definable Galois extension of finite type of $A$ relative to $X$. 
    \end{enumerate}
\end{definition}

\begin{lemma}\label{lem: mcl points}
    Let $G$ be an $A$-definable group. Let $P$ be an $A$-definable PHS for $G$. If $G(M) = G(\mcl(A))$ then $P(M) = P(\mcl(A))$. 
\end{lemma}

\begin{proof}
    By definition $P$ is nonempty and therefore has a point $p$ in $M$. Let $\gamma:M \to M$ be any $A$-elementary embedding. $\gamma(M) \prec M$ must also have a realization of $P$, $p_1$. By assumption $G(M) = G(\mcl(A)) = G(\gamma(M))$. Both $p$ and $p_1$ are realizations of $P$ in $M$ and so there is a $g \in G(\mcl(A))$ with $p_1g = p$. Hence $p \in \dcl(p_1,g) \subseteq \gamma(M)$. Since $\gamma$ was arbitrary, $p \in \mcl(A)$.
\end{proof}

\begin{remark}
    If $G$ is strongly minimal or lives on a strongly minimal set then $G(M) = G(\mcl(A))$. For example, let $K$ be a differential field of characteristic $0$ and the ambient theory be the theory of differentially closed fields of characteristic $0$, then one can take $G$ to be an algebraic group defined over the constants $C_K$ and the above condition will be satisfied.
\end{remark}

Lemma \ref{lem: mcl points} has the following consequences:

\begin{theorem}\label{thm: mcl points}
    Let $G$ be an $A$-definable group with $G(M) = G(\mcl(A))$. Then
    \begin{enumerate}
        \item The definable Galois cohomology of $G$ over $\mcl(A)$ is trivial: $$\HH^1_{\defin}(M/\mcl(A),G(M)) = 1.$$
        \item Using the short exact sequence $$\HH^1_{\defin}(\mcl(A)/A,G(\mcl(A))) \cong \HH^1_{\defin}(A,G)$$ Namely, all of the definable Galois cohomological information for $G$ in $M$ is contained in $\mcl(A)$.
        \item Any definable Galois extension relative to $G$ is contained in $\mcl(A)$.
    \end{enumerate}
\end{theorem}

\begin{proof}
    \begin{enumerate}
        \item Let $P$ be an $A$-definable right PHS for $G$. By Lemma \ref{lem: mcl points}, $P$ has an $mcl(A)$ point so by definition $\HH^1_{\defin}(M/\mcl(A),G(M)) = 1$.
        \item By Theorem \ref{thm: SES}, the intermediate extension $A \subseteq \mcl(A) \subseteq M$ induces a short exact sequence in cohomology $$1 \to \HH^1_{\defin}(\mcl(A)/A,  G(\mcl(A))) \to \HH^1_{\defin}(A, G) \to \HH^1_{\defin}(M/\mcl(A), G(M)).$$ As the last term of the above sequence is $1$ we have $$\HH^1_{\defin}(\mcl(A)/A,G(\mcl(A))) \cong \HH^1_{\defin}(A,G).$$
        \item Let $A \subseteq B = \dcl(A,b)$ be a definable Galois extension relative to $G$. Let $(Q,H)$ be the associated $A$-definable torsor. As $H$ lives on $G$, $H(M) = H(\mcl(A))$. Then by Lemma \ref{lem: mcl points}, $Q(M) = Q(\mcl(A))$. As $B$ is generated by a $M$-point of $Q$, $B$ is contained in $\mcl(A)$. This is a slight strengthening of Fact \ref{fact: mcl} (3).
    \end{enumerate}
\end{proof}

The next proposition generalizes the relationships between the series Picard-Vessiot envelopes, the Picard-Vessiot closure, the minimal closure, and differential closure of a differential field $K \subseteq K^{PV_1} \subseteq K^{PV_2}  \cdots \subseteq  K^{PV_\infty}$ as in \cite{meretzky2026galoistheoryautomorphismgroups}. 

\begin{proposition}\label{prop: struct of M}
    Let $A$ be a $\dcl$-closed set of parameters. Let $X$ be an $A$-definable set such that $X(M) = X(\mcl(A))$. Then each $A^{X\text{DG}_n}$ is contained in $\mcl(A)$ and we have $$A \subseteq A^{X\text{DG}_1} \subseteq A^{X\text{DG}_2} \cdots \subseteq A^{X\text{DG}_\infty} \subseteq \mcl(A) \subseteq M.$$ Furthermore each $A^{X\text{DG}_n}$ is normal. 
\end{proposition}

\begin{proof}
    This follows from Theorem \ref{thm: mcl points} (3). Normality follows because each conjugate under $\aut(M/A)$ of an $X$-definable Galois extension is again an $X$-definable Galois extension.
\end{proof}

We now strengthen the assumption $X(M)=X(\mcl(A))$ to the assumption that $X(M)=X(A)$. This is a generalization of the assumption that the field of constants of the base field be algebraically closed in the context of Picard-Vessiot differential Galois theory. It has the consequence that each $X$-definable Galois extension is then normal. An immediate consequence is that the family of $X$-definable Galois extensions and the family of normal infinite $X$-iterated definable Galois extensions form directed systems:

\begin{lemma}\label{lem: IDG form directed system}
    Let $B$ and $C$ be iterated definable Galois extensions of $A$ of finite type relative to $X$ in $M$. Then $\dcl(B,C)$ is again an iterated definable Galois extension of $A$ of finite type relative to $X$ in $M$.
\end{lemma}

\begin{proof}
    We follow Lemma 2 of \cite{magid2022completepicardvessiotclosure}. Let $A=B_0 \subseteq B_1 \subseteq \cdots B_m=B$ and $A=C_0 \subseteq C_1 \subseteq \cdots C_n=C$ witness that $B$ and $C$ are iterated definable Galois extensions of finite type. The claim is proved by induction on $n$. If $n$ is $0$ then $C=A$ and $\dcl(B,C) = B$. Assume true for $n-1$. So $\dcl(B_m,C_{n-1})$ is an iterated definable Galois extension of $A$. By assumption $C_n = \dcl(C_{n-1},c_n)$ where $tp(c_n/C_{n-1})$ is weakly orthogonal to $X$ and strongly internal to $X$. Then by Fact \ref{weak orth fact}, $\tp(c_n/B_m,C_{n-1})$ is still weakly orthogonal to $X$ as $X(M) = X(A)$. Adding parameters preserves strong internality to $X$. So $\dcl(B,C)$ is a definable Galois extension of $\tp(c_n/B_m,C_{n-1})$ of finite type. So $\dcl(B,C)$ is still an iterated definable Galois extension of finite type.
\end{proof}

\begin{lemma}\label{lem: finite type iterated is dg}
    Let $X(M)=X(A)$. A normal iterated definable Galois extension $B$ of $A$ of finite type relative to $X$ must be a definable Galois extension of finite type
\end{lemma}

\begin{proof}
    Let $b$ $\dcl$-generate $B$ over $A$. Then by normality and strong homogeneity any other realization of the type of $b$ over $A$ must be contained in $\dcl(A,b,X(A))$. Again as $X(M)=X(A)$, $\tp(b/A)$ is weakly orthogonal to $X$ by Fact \ref{weak orth fact}. 
\end{proof}

\begin{proposition}\label{prop: idg colim}
    Let $G$ be an $A$ definable group living on an $A$-definable set $X$. Assume $X(M)=X(A)$. Then $$\varinjlim \HH^{1}_{\defin}(B/A, G(B)) \cong \HH^{1}_{\defin}(A^{XDG_1}/A, G(A^{XDG_1}))$$ where the limit is taken over the family of normal iterated definable Galois extensions of finite type $B$ of $A$ relative to $X$, equivalently, the family of definable Galois extensions of finite type.
\end{proposition}

\begin{proof}
    By Lemma \ref{lem: finite type iterated is dg}, a normal iterated definable Galois extension of finite type $B$ of $A$ relative to $X$ is a definable Galois extension of finite type. Then by Lemma \ref{lem: IDG form directed system}, the family of definable Galois extensions of finite type $B$ of $A$ relative to $X$ form a directed system. It is clear that their union is $A^{XDG_1}$. The result then follows from Lemma \ref{lem: colim}. 
\end{proof}

As in the PV setting of the previous section, to obtain a colimit representation of $\HH^{1}_{\defin}(A^{XDG_\infty}/A, G(A^{XDG_\infty}))$ we must then work with the family of normal closures of iterated definable Galois extensions of finite type $B$ of $A$ relative to $X$.

\begin{theorem}\label{thm: idg colim}
    Let $G$ be an $A$ definable group living on an $A$-definable set $X$. Assume $X(M)=X(A)$. Then $$\varinjlim \HH^{1}_{\defin}(B/A, G(B)) \cong \HH^{1}_{\defin}(A^{XDG_\infty}/A, G(A^{XDG_\infty}))$$ where the limit is taken over the family of normal closures of iterated definable Galois extensions of finite type $B$ of $A$ relative to $X$.
\end{theorem}

\begin{proof}
    By Lemma \ref{lem: IDG form directed system}, the family of normal closures of iterated definable Galois extensions of finite type $B$ of $A$ relative to $X$ form a directed system. It is clear that their union is $A^{XDG_\infty}$. The result then follows from Lemma \ref{lem: colim}. 
\end{proof}

We then propose the following definition of a bounded set of parameters $A$.  

\begin{definition}
    Let $X$ be an $A$-definable set satisfying $X(A)= X(M)$. Let $G$ be an $A$-definable group living on $X$.  The set of parameters $A$ is said to be bounded with respect to $G$ and $X$ if there exists an iterated definable Galois extension of finite type relative to $X$ whose normal closure $B$ satisfies $$\HH^{1}_{\defin}(B/A, G(B)) \cong \HH^{1}_{\defin}(A^{XDG_\infty}/A, G(A^{XDG_\infty})).$$
\end{definition}

This definition is a generalization of Definition \ref{def: diff bdd}.  
However by Propositions \ref{thm: mcl points} and \ref{prop: idg colim}, a general colimit formula and a definition of boundedness for all of $\HH^1_{\defin}(M/A,G(M))$, i.e., with reference to the extension $M/A$, would require and understanding of the difference between $\HH^{1}_{\defin}(A^{XDG_\infty}/A, G(A^{XDG_\infty}))$ and  $\HH^{1}_{\defin}(\mcl(A)/A, G(\mcl(A)))$. This is sensitive to the theory chosen and it is false in general that they are isomorphic:
        
In $DCF_0$ and the linear case of $\DCF_{0,m}$, the results of \cite{PILLAY2017809}, \cite{Minchenko_2019}, and \cite{Chatzidakis2017GeneralizedPE} give the isomorphism in various cases. We note that difference algebra falls outside of the model theoretic totally transcendental setting of this paper. However, results of \cite{Chatzidakis2017GeneralizedPE} and \cite{Bachmayr_Wibmer_2026} give appropriate isomorphisms in this setting. Generalizing the definable cohomological machinery to coincide with the Galois cohomology measuring difference algebraic torsors defined by Bachmayr and Wibmer is also left to future work.

%\newpage
\bibliography{main.bib}{}
\bibliographystyle{plain}
\nocite{*}

\end{document}